\documentclass[12pt,reqno]{amsart}

\usepackage{amssymb,amsmath,amsbsy,eucal,amsfonts,mathrsfs,latexsym}
\usepackage{graphicx,verbatim,psfrag,epsfig,enumerate}
\usepackage{stmaryrd}
\usepackage{lscape}
\usepackage{subfig}
\usepackage{multirow}

\newtheorem{theorem}{Theorem}[section]
\newtheorem{lemma}[theorem]{Lemma}

\theoremstyle{definition}

\newtheorem{remark}{{\it Remark}}[section]

\numberwithin{equation}{section}

\newcommand{\nn}{\nonumber}

\newcommand{\Vc}{{\mathbf{c}}}
\newcommand{\Vd}{{\mathbf{d}}}

\newcommand{\Vi}{{\mathbf{i}}}

\newcommand{\Vl}{{\mathbf{l}}}

\newcommand{\Vr}{{\mathbf{r}}}

\newcommand{\Vu}{{\mathbf{u}}}
\newcommand{\Vv}{{\mathbf{v}}}

\newcommand{\Be}{{\boldsymbol{e}}}
\newcommand{\Bf}{{\boldsymbol{f}}}
\newcommand{\Bg}{{\boldsymbol{g}}}

\newcommand{\Bn}{{\boldsymbol{n}}}

\newcommand{\Br}{{\boldsymbol{r}}}

\newcommand{\Bu}{{\boldsymbol{u}}}
\newcommand{\Bv}{{\boldsymbol{v}}}
\newcommand{\Bw}{{\boldsymbol{w}}}

\newcommand{\BH}{{\boldsymbol{H}}}

\newcommand{\BJ}{{\boldsymbol{J}}}

\newcommand{\BL}{{\boldsymbol{L}}}
\newcommand{\BM}{{\boldsymbol{M}}}

\newcommand{\BP}{{\boldsymbol{P}}}

\newcommand{\BU}{{\boldsymbol{U}}}
\newcommand{\BV}{{\boldsymbol{V}}}

\newcommand{\Ce}{{\mathcal E}}

\newcommand{\Ct}{{\mathcal T}}

\newcommand{\curl}{\Vc\Vu\Vr\Vl\,}
\newcommand{\divv}{\Vd\Vi\Vv \,}
\newcommand{\im}{\mathbf{i}}

\begin{document}

\title{An absolutely stable $hp$-HDG method for 
the time-harmonic Maxwell equations with high wave number}


\author{Peipei Lu}
\address{School of Mathematics Sciences, Soochow University, Suzhou, 215006, China}
\email{pplu@suda.edu.cn}

\author{Huangxin Chen}
\address{School of Mathematical Sciences and Fujian Provincial Key Laboratory on Mathematical Modeling and 
High Performance Scientific Computing, Xiamen University, Fujian, 361005, China}
\email{chx@xmu.edu.cn}

\author{Weifeng Qiu}
\address{Department of Mathematics, City University of Hong Kong, 83 Tat Chee Avenue, Kowloon, Hong Kong, China}
\email{weifeqiu@cityu.edu.hk}
\thanks{Corresponding author: Weifeng Qiu}

\thanks{All authors contributed equally in this paper.
The work of the first author was supported by the NSF of China (Grant No.11401417), the Program of  
Natural Science Research of Jiangsu Higher Education Institutions of China (Grant No. 14KJB110021)
and Jiangsu Provincial Key Laboratory for Numerical Simulation of Large Scale Complex Systems (No. 201404). 
The work of the second author was supported by the NSF of China (Grant No. 11201394) and the Fundamental 
Research Funds for the Central Universities (Grant No. 20720150005). The work of the third author was partially 
supported by a grant from the Research Grants Council of the Hong Kong Special Administrative Region, China 
(Project No. CityU 11302014). }


\begin{abstract}
We present and analyze a hybridizable discontinuous Galerkin (HDG) method for the time-harmonic Maxwell equations. 
The divergence-free condition is enforced on the electric field, then a Lagrange multiplier is introduced, and the problem 
becomes the solution of a mixed curl-curl formulation of the Maxwell's problem. The method is shown to be an absolutely 
stable HDG method for the indefinite time-harmonic Maxwell equations with high wave number. By exploiting the duality 
argument, the dependence of convergence of the HDG method on the wave number $\kappa$, the mesh size $h$ and 
the polynomial order $p$ is obtained. Numerical results are given to verify the theoretical analysis.
\end{abstract}



\maketitle

\markboth{P. Lu, H. Chen and W. Qiu}{HDG method for Maxwell equations}

\section{Introduction}\label{introduction}
The time-harmonic Maxwell boundary value problem reads as follows:
\begin{subequations}
\label{pde_original}
\begin{align}
\label{pde_original_1}
\curl\curl\Bu - \kappa^2 \Bu &= \widetilde{\Bf} \qquad \rm{in}\ \Omega,\\
\label{BC-PDE}
\curl \Bu \times \Bn + \im \kappa \Bu^t &= \widetilde{\Bg} \qquad \rm{on}\ \partial\Omega,
\end{align}
\end{subequations}
where $\Omega\subset \mathbb{R}^3$ is a bounded, uniformly star-shaped polyhedral domain, 
the wave number $\kappa$ is real and positive, $\im$ denotes the imaginary unit, $\Bn$ denotes the unit outward normal 
to $\partial \Omega$, and $\Bu^t = ( \Bn \times \Bu ) \times \Bn$ denotes the tangential component of the electric field 
$\Bu$. Equation (\ref{BC-PDE}) is the standard impedance boundary condition which requires $\widetilde{\Bg} \cdot \Bn = 0$, 
thus, $\widetilde{\Bg}^t = \widetilde{\Bg}$. The above Maxwell equations are of considerable importance in the engineering 
and scientific computation. In this paper we assume the current density is divergence-free (namely $\divv\widetilde{\Bf} =0$), 
hence the electric field $\Bu$ is also divergence-free. 

The Maxwell's operator is strongly indefinite for high wave number $\kappa$, which brings difficulties both in 
theoretical analysis and numerical simulation. Various numerical methods which include finite element methods 
(FEM) \cite{Nedelec80,Nedelec86,Hiptmair02,Hiptmair11,Brenner07,Zhong09}, discontinuous Galerkin (DG) 
methods \cite{Perugia02,Perugia03, Cockburn04,Houston04,Houston05,Nguyena11,R.Hiptmair13,FW2014} 
and weak Galerkin FEM method \cite{MWYS15} have been developed to solve the Maxwell's problem. In particular, 
Feng and Wu \cite{FW2014} recently proposed and analyzed an interior penalty discontinuous Galerkin (IPDG) method 
for the problem (\ref{pde_original}) with high wave number, which is uniquely solvable without any mesh constraint. 
DG methods have several attractive features which include the capabilities to handle complex geometries, to provide 
high-order accurate solutions, etc. But the dimension of approximation DG space is much larger than the 
dimension of the corresponding conforming space. Hybridizable discontinuous Galerkin (HDG) methods \cite{Cockburn09} 
were recently introduced to address this issue. The HDG methods retain the advantages of standard DG methods, 
and the resulting system is only due to the unknowns on the skeleton of the mesh.

Two HDG methods were presented in \cite{Nguyena11} for the numerical solution of the Maxwell problem. 
The first HDG method enforces the divergence-free condition on the electric field and introduces a Lagrange 
multiplier. It produces a linear system for the degrees of freedom (DOF) of the approximate traces of both the 
tangential component of the vector field and the Lagrange multiplier. The second HDG method does not enforce 
the divergence-free condition and results in a linear system only for the DOF of the approximate trace of the 
tangential component of the vector field. Compared to the IPDG method for the time-harmonic Maxwell equations 
in \cite{Houston05,FW2014}, the two HDG methods have less globally coupled unknowns. The well-posedness, 
conservativity and consistence of the two HDG methods, together with a numerical demonstration, have been shown 
in \cite{Nguyena11}. However, no convergence 
analysis is given in \cite{Nguyena11}. Recently, the $h$-convergence analysis of the second HDG method was considered in \cite{FLX2015}. In this paper we are interested in the $hp$-convergence analysis for the first HDG method 
mentioned in \cite{Nguyena11} which solves a mixed curl-curl formulation of the time-harmonic Maxwell equation
\begin{subequations}
\label{pde_mixed}
\begin{align}
\label{maxwell_p}
\curl\curl\Bu - \kappa^2 \Bu + \nabla \widetilde{\sigma}&= \widetilde{\Bf} \qquad \rm{in}\ \Omega,\\
\divv \Bu &= 0 \qquad \rm{in}\ \Omega,\\
\label{BC-PDE-1}
\curl \Bu \times \Bn + \im \kappa \Bu^t &= \widetilde{\Bg} \qquad \rm{on}\ \partial\Omega, \\
\label{BC_p}
\widetilde{\sigma}&= 0 \qquad \rm{on}\ \partial\Omega,
\end{align}
\end{subequations}
where $\widetilde{\sigma}$ is a scalar Lagrange multiplier used to enforce the divergence-free condition. 
Taking the divergence of the equation (\ref{maxwell_p}) yields $ \Delta \widetilde{\sigma} = 0$, which 
together with the boundary condition (\ref{BC_p}) implies that $\widetilde{\sigma} = 0$ throughout the 
domain. Hence, under the divergence-free condition of the current density, the equations (\ref{pde_original}) 
and (\ref{pde_mixed}) are equivalent.

We aim to develop an HDG method which is absolutely stable without any mesh constraint for the above mixed 
curl-curl formulation (\ref{pde_mixed}) and reveal the dependence of convergence for the HDG method on the wave 
number $\kappa$, the mesh size $h$ and the polynomial order $p$. We mention that only simple $L^2$-projections 
are used in our analysis which is different from the projection-based error analysis in \cite{Cockburn10}, and 
the $p$-dependence of the stability estimate and the convergence can be derived. 
We also mention that the stabilization parameters in our HDG method are different from those in \cite{Nguyena11}. The focus of our analysis is to apply the duality argument to establish the rigorous stability estimate and error analysis for the HDG method proposed for the mixed curl-curl formulation (\ref{pde_mixed}). Intrinsically, the regularity estimate of the solution of the dual problem used in this paper can be obtained due to introduction of a Lagrange multiplier in the mixed curl-curl formulation. This is also the reason why the Helmholtz decomposition technique can be avoided in the analysis and the $p$-estimate can be derived. We first apply the duality argument to obtain the estimates for $\|\Bu_h\|_{\Ct_h}$ and the error $\|\Bu-\Bu_h\|_{\Ct_h}$, then the estimates for other variables of the HDG method can be further obtained. Up to our best knowledge, we give the first $p$-estimate of numerical methods using piecewise polynomial solution spaces for solving the time-harmonic Maxwell equations with high wave number.

The remainder of this paper is the following. We give some notations, introduce the HDG method for the mixed curl-curl formulation of the time-harmonic Maxwell equations (\ref{pde_mixed}) and present the main results of stability estimates and error estimates in the next section. Section \ref{stability_estimate} and section \ref{error_analysis} are devoted to providing detailed proofs of the stability estimates and error estimates respectively. In section \ref{ideal_case}, we discuss the stability estimates and error estimates for the HDG method under some ideal assumptions of the problem (\ref{pde_original}) and the associated dual problem. In the final section, we give some numerical results to confirm our theoretical analysis.

\section{Notation, HDG method and main results}\label{HDG_sec2}

Let $\Bf:=-\Vi \widetilde{\Bf}, \sigma:=-\Vi \widetilde{\sigma}$ and $\Bg := -\Vi \widetilde{\Bg}$.
The HDG scheme for the equation (\ref{pde_mixed}) is based on a first-order system of this equation, 
which can be rewritten in a mixed formulation as follows:
\begin{subequations}
\label{pde_mixed_first_order}
\begin{align}
\Vi \Bw-\curl \Bu &= 0 \qquad \rm{in}\ \Omega,\\
\curl\Bw + \Vi \kappa^2 \Bu + \nabla \sigma&= \Bf \qquad \rm{in}\ \Omega,\\
\divv \Bu &= 0 \qquad \rm{in}\ \Omega,\\
\Bw \times \Bn +  \kappa \Bu^t &= \Bg \qquad \rm{on}\ \partial\Omega, \\
\sigma&= 0 \qquad \rm{on}\ \partial\Omega.
\end{align}
\end{subequations}

Throughout the paper we use the standard notations and definitions for Sobolev spaces (see, e.g., Adams \cite{Adams}). 
We denote by $\Ct_h$ a conforming triangulation of $\Omega$ made of shape-regular simplicial elements. 
We denote by $h_T$ the diameter of $T \in \Ct_h$ and $h = \max_{T \in \Ct_h} h_T$, the collection of faces 
is denoted by $\Ce_h$, with the collection of interior faces by $\Ce^0_h$ and the collection of boundary faces 
by $\Ce^\partial_h$, the collection of element boundaries by $\partial \Ct_h :=\{ \partial T | T \in \Ct_h \}$. 
We let $C$ denote a positive number independent of the mesh size, polynomial order and wave number, whose value can take on different values in different occurrences. The corresponding finite element spaces for the HDG 
method for the first-order system (\ref{pde_mixed_first_order}) are defined to be
\begin{align*}
\BV_h &: = \{ \Br \in \BL^2(\Omega) \, : \,  \Br|_T \in \BP_p(T), \forall \, T \in \Ct_h  \},\\
\BU_h &: = \{ \Bv \in \BL^2(\Omega) \, : \,  \Bv|_T \in \BP_p(T), \forall \, T \in \Ct_h  \},\\
\BM^t_h &: = \{ \boldsymbol{\eta} \in \BL^2(\Ce_h) \, : \,   \boldsymbol{\eta}|_F \in \BP_p(F),   
(\boldsymbol{\eta} \cdot \Bn)|_F=0,   \forall \, F \in \Ce_h  \},\\
Q_h &: = \{ q \in L^2(\Omega) \, : \, q|_T \in P_p(T) ,  \forall \, T \in \Ct_h \},\\
M_h &:= \{\xi  \in L^2(\Ce_h) \, : \, \xi|_F \in P_p(F), \forall \, F \in \Ce_h \},
\end{align*}
where the polynomial order $p\geq 1$, $\BL^2(\Omega) = [L^2(\Omega)]^3$, $\BL^2(\Ce_h) = [L^2(\Ce_h)]^3$, 
$ \BP_p(T)=[ P_p(T)]^3$ and $ \BP_p(F) =  [P_p(F)]^3$. Here, $P_p(D)$ denotes the space of complex-valued polynomials of degree at most $p$ on $D$. Let $P_M$ denote the standard $L^2$-projection operator from 
$L^2(\Ce_h)$ onto $P_p(\Ce_h)$. In addition, we set $M_h(g):=\{ \xi \in M_h \, : \, 
\xi = P_M g  \  \rm{on}\ \partial\Omega\}$.  Similarly,  $\BP_{\BM}$ denotes the standard $L^2$-projection operator from 
$\BL^2(\Ce_h)$ onto $\BP_p(\Ce_h)$. We use $\boldsymbol{\Pi_V}, \boldsymbol{\Pi_U},\Pi_Q$ to denote the 
standard $L^2$-projection onto $\BV_h,\BU_h$ and $Q_h$ respectively. In the analysis, we shall use the following 
approximation results of $L^2$-projections:
\begin{subequations}
\begin{align} \label{es_pj_1}
\| \Bw-\boldsymbol{\Pi_V}\Bw \|_{\Ct_h}& \leq C{h}^t/{p}^t \| \Bw \|_{t,\Omega}  \qquad   0\leq t\leq p+1,\\
\label{es_pj_2}
\| \Bu-\boldsymbol{\Pi_U}\Bu \|_{\Ct_h} &\leq C{h}^s/{p}^s  \| \Bu \|_{s,\Omega} \qquad 0\leq s\leq p+1,\\
\label{es_pj_3}
\| \sigma-{\Pi_Q}\sigma \|_{\Ct_h} &\leq C{h}^\beta/{p}^\beta \| \sigma \|_{\beta,\Omega} \qquad 0\leq \beta\leq p+1,\\
\label{es_pj_4}
\| \Bw-\boldsymbol{\Pi_V}\Bw \|_{0,\partial T} &\leq C{h}^{t-\frac{1}{2}}/{p}^{t-\frac{1}{2}} \| \Bw \|_{t,T} 
\qquad \forall T \in \Ct_h,\ 0\leq t\leq p+1,\\
\label{es_pj_5}
\| \Bu-\boldsymbol{\Pi_U}\Bu \|_{0,\partial T} &\leq C{h}^{s-\frac{1}{2}}/{p}^{s-\frac{1}{2}} \| \Bu \|_{s,T} 
\qquad \forall T \in \Ct_h,\ 0\leq s\leq p+1,\\
\label{es_pj_6}
\| \sigma-{\Pi_Q}\sigma \|_{0,\partial T} &\leq C{h}^{\beta-\frac{1}{2}}/{p}^{\beta-\frac{1}{2}} \| \sigma \|_{\beta,T} 
\qquad \forall T \in \Ct_h,\ 0\leq \beta\leq p+1,\\
\label{es_pj_7}
\| \Bw-\boldsymbol{P_M}\Bw \|_{\partial \Ct_h} &\leq C{h}^{t-\frac{1}{2} }/{p}^{t-\frac{1}{2} } \| \Bw \|_{t,\Omega} 
\qquad 0\leq t\leq p+1,\\
\label{es_pj_8}
\| \Bu-\boldsymbol{P_M}\Bu \|_{\partial \Ct_h} &\leq C{h}^{s-\frac{1}{2} }/{p}^{s-\frac{1}{2} } \| \Bu \|_{s,\Omega} 
\qquad 0\leq s\leq p+1,\\
\label{es_pj_9}
\| \sigma-{P_M}\sigma \|_{\partial \Ct_h} &\leq C{h}^{\beta-\frac{1}{2}}/{p}^{\beta-\frac{1}{2}} \| \sigma \|_{\beta,\Omega} 
\qquad 0\leq \beta\leq p+1.
\end{align}
\end{subequations}
Here $\|\cdot\|_{\Ct_h} = ( \sum_{T\in \Ct_h} \|\cdot\|^2_{0,T} )^{\frac{1}{2}}$ and $\|\cdot\|_{\partial \Ct_h} = ( \sum_{T\in \Ct_h} \|\cdot\|^2_{0,\partial T} )^{\frac{1}{2}}$. The above results hold due to the $hp$ approximation theory of polynomials and trace inequality when $\Ct_h$ 
consists of shape-regular simplices (cf. \cite{Schwab98,FW2011,Chernov2012,Egger13,Melenk14,Cangiani14}). 
The above $h$-dependence approximation results hold when $\Ct_h$ consists of shape-regular polyhedral elements. 
Thus when we only consider $\kappa$- and $h$-dependence in our analysis, $\Ct_h$ can be a conforming mesh 
consisting of shape-regular polyhedral elements. This is due to the fact that only the approximation results (\ref{es_pj_1})-(\ref{es_pj_3}) have been deduced recently in the literature (cf. \cite{Cangiani14}) when the mesh 
consists of general polyhedral elements. The $p$-dependence of convergence for the trace estimate of the polynomial 
$L^2$-projection (cf. (\ref{es_pj_4})-(\ref{es_pj_6})) was first studied in \cite{Chernov2012} on simplicial element, and as far as we know, no extension of the estimates (\ref{es_pj_4})-(\ref{es_pj_6}) to the $L^2$-projection defined on the general polyhedral element has been obtained.

We define the bilinear forms
\begin{align*}
(\boldsymbol{\eta},\boldsymbol{\zeta})_{\mathcal T_h}:=\sum_{T\in \mathcal
T_h}(\boldsymbol{\eta},\boldsymbol{\zeta})_{T}, \quad \langle \boldsymbol{\eta},\boldsymbol{\zeta} 
\rangle_{  \partial \mathcal T_h} := \sum_{T\in \mathcal T_h}\langle \boldsymbol{\eta},\boldsymbol{\zeta}
\rangle_{\partial T}\\ (\eta,\zeta)_{\mathcal T_h}:=\sum_{T\in \mathcal
T_h}(\eta,\zeta)_T,  \quad \ \langle \eta,\zeta \rangle_{\partial \mathcal
T_h}:=\sum_{T\in \mathcal T_h}\langle \eta,\zeta\rangle_{\partial T},
\end{align*}
where $(\boldsymbol{\eta},\boldsymbol{\zeta})_{D}$ (respectively, $(\eta,\zeta)_{D}$) denotes the integral of 
$\boldsymbol{\eta}\cdot \overline{\boldsymbol{\zeta}}$ (respectively, $\eta\overline{\zeta}$) over 
$D \subset \mathbb{R}^3$ and $\langle\boldsymbol{\eta},\boldsymbol{\zeta}\rangle_{D}$ (respectively, 
$\langle\eta,\zeta\rangle_{D}$) denotes the integral of $\boldsymbol{\eta}\cdot\overline{\boldsymbol{\zeta}}$ 
(respectively, $\eta\overline{\zeta}$) over $D \subset  \mathbb{R}^2$.

The HDG method for the first-order system (\ref{pde_mixed_first_order}) yields a solution 
$(\Bw_h,\Bu_h,\widehat{\Bu}^t_h,\sigma_h, \widehat{\sigma}_h) \in \BV_h \times \BU_h 
\times \BM^t_h\times  Q_h \times M_h(0)$ such that
\begin{subequations}
\label{discrete_mixed_form}
\begin{align}
\label{discrete_mixed_form_a}
&(\Vi \Bw_h,\Br_h)_{\Ct_h} - (\Bu_h,\curl \Br_h)_{\Ct_h} + \langle\widehat{\Bu}^t_h\times \Bn, \Br_h \rangle_{\partial \Ct_h}  =0 \\
\label{discrete_mixed_form_b}
&(\Bw_h,\curl \Bv_h)_{\Ct_h} - \langle \widehat{\Bw}^t_h \times \Bn,\Bv_h \rangle_{\partial \Ct_h} 
+ ( \Vi \kappa^2 \Bu_h,\Bv_h )_{\Ct_h}\nn\\
&\qquad\qquad\qquad  \ \ - ( \sigma_h,\divv \Bv_h )_{\Ct_h} + \langle \widehat{\sigma}_h, \Bv_h 
\cdot \Bn \rangle_{\partial \Ct_h} = (\Bf, \Bv_h)_{\Ct_h},\\
\label{discrete_mixed_form_c}
&-(\Bu_h,\nabla q_h)_{\Ct_h} + \langle \widehat{\Bu}^n_h \cdot \Bn, q_h \rangle_{\partial \Ct_h} = 0,\\
\label{discrete_mixed_form_d}
&\langle \widehat{\Bw}^t_h \times \Bn, \boldsymbol{\eta}_h \rangle_{\partial \Ct_h \setminus \partial \Omega}=0,\\
\label{discrete_mixed_form_e}
& \langle \widehat{\Bw}^t_h \times \Bn, \boldsymbol{\eta}_h \rangle_{\partial \Omega} 
+ \langle \kappa \widehat{\Bu}^t_h, \boldsymbol{\eta}_h \rangle_{\partial \Omega} 
= \langle \Bg,  \boldsymbol{\eta}_h \rangle_{\partial \Omega},\\
\label{discrete_mixed_form_f}
& \langle  \widehat{\Bu}^n_h \cdot \Bn, {\xi}_h\rangle_{\partial \Ct_h} = 0,
\end{align}
\end{subequations}
for all $(\Br_h, \Bv_h, \boldsymbol{\eta}_h, q_h,\xi_h) \in \BV_h \times \BU_h 
\times \BM^t_h \times  Q_h \times M_h(0)$, where
\begin{align}
\label{hat_definition}
\widehat{\Bw}_h  = \Bw_h + \tau_t(\Bu^t_h-\widehat{\Bu}^t_h)\times \Bn, 
\quad \widehat{\Bu}^n_h = {\Bu}^n_h + \tau_n(\sigma_h-\widehat{\sigma}_h) \Bn.
\end{align}
Here, for any vector $\Br \in \mathbb{R}^3$, $\Br^n = (\Br \cdot \Bn) \Bn$ denotes the normal component of 
the vector $\Br$. The parameters $\tau_t$ and $\tau_n$ are the so-called {\it local stabilization parameters} which have an
important effect on both the stability of the solution and the accuracy of the HDG scheme. We choose $\tau_t = p/h$ and $\tau_n = (1+\kappa) h/p$ in this paper.

\begin{remark}
The mixed curl-curl formulation (\ref{pde_mixed}) can also be applied to the Maxwell equations (\ref{pde_original}) 
with $\divv \Bf \neq 0$. In this case $\divv \Bu = {\theta} \neq 0$ with ${\theta}$ a given variable. 
Indeed, taking the divergence of the equation (\ref{pde_original_1}) implies that ${\theta}$ satisfies that 
$-\kappa^2 {\theta} = \divv \widetilde{\Bf}$. Then taking the divergence of the equation (\ref{maxwell_p}) 
again yields $ \Delta \widetilde{\sigma} =\divv \widetilde{\Bf} + \kappa^2{\theta} =  0$, which together with 
the boundary condition (\ref{BC_p}) also implies that $\widetilde{\sigma} = 0$. Hence, the HDG scheme in this paper can 
also be used for the Maxwell equations (\ref{pde_original}) with $\divv \widetilde{\Bf} \neq 0$. However, we mention that if the HDG method is used with non divergence-free current density, the regularity estimates in \cite{Hiptmair11} can not be applied. Thus, the theoretical analysis throughout this paper holds only under the assumption of divergence-free current density in (\ref{pde_original_1}).
\end{remark}

When $\Bf$ is divergence-free and $\Bg\in
\BH_T^{\frac{1}{2}}(\partial\Omega)$, the solution of the first-order system (\ref{pde_mixed_first_order}) satisfies that $\Bu \in \BH^{\frac{1}{2}+\alpha}(\Omega)$ and $\Bw \in \BH^{\frac{1}{2}+\alpha}(\Omega)$, and there holds (cf. \cite{Hiptmair11})
\begin{align}
\kappa \| \Bu \|_{\frac{1}{2}+\alpha,\Omega} + \|\Bw\|_{\frac{1}{2}+\alpha,\Omega} \leq C ( 1 + \kappa ) \BM(\Bf,\Bg) +C \|\Bg\|_{\frac{1}{2},\partial \Omega}. \label{sta_weak_new}
\end{align}
To state our main results, we need a regularity assumption of the dual problem. Let $\boldsymbol{\Psi}$ and $\varphi$ be 
the solution of the following dual problem:
\begin{subequations}
\label{dual_problem}
\begin{align}
\label{dual_p_1}
\curl \curl \boldsymbol{\Psi} - \kappa^2 \boldsymbol{\Psi} + \Vi \nabla \varphi &= \Vi \BJ  \quad  \ \rm{in}\ \Omega,\\
\label{dual_p_2}
\divv \boldsymbol{\Psi} &= 0\qquad \rm{in}\ \Omega,\\
\label{bc_dual_3}
\curl \boldsymbol{\Psi} \times \Bn - \Vi \kappa \boldsymbol{\Psi}^t &= 0 \qquad \rm{on}\ \partial\Omega,\\
\varphi &= 0 \qquad \rm{on}\ \partial\Omega,
\end{align}
\end{subequations}
where $\forall \BJ \in \BU_h \subset \BL^2(\Omega)$. Due to the fact that $\Omega$ is a bounded, uniformly star-shaped polyhedral domain, the solution $(\boldsymbol{\Psi}, \varphi)$ has 
the following regularity estimate (cf. \cite{Hiptmair11,FW2014}):
\begin{align}
\label{est_dual_1}
\kappa \|\boldsymbol{\Psi}\|_{\frac{1}{2}+\alpha,\Omega} + \|\curl \boldsymbol{\Psi}\|_{\frac{1}{2}+\alpha,\Omega} 
+ \kappa(1+\kappa) \| \boldsymbol{\Psi}\|_{0,\Omega}  +(1+ \kappa) \|\nabla \varphi\|_{0,\Omega} 
\leq C (1+\kappa) \| \BJ\|_{0,\Omega},
\end{align}
where $0<\alpha\leq \widetilde{\alpha}$ and $0<\widetilde{\alpha}<\frac{1}{2}$ is a parameter only depending on $\Omega$. When $\Omega$ is convex, the above estimate holds true for all $0<\alpha<\frac{1}{2}$.  Moreover, when $\Omega$ is also a $C^2$ star-shaped domain, (\ref{est_dual_1}) holds true for $\alpha=\frac{1}{2}$. In the following, we show that (\ref{est_dual_1}) holds true under the assumption of the domain $\Omega$. It is easy to see that $\boldsymbol{\Psi}$ satisfies 
\begin{align}
\label{dual_proof_p_1}
\curl \curl \boldsymbol{\Psi} - \kappa^2 \boldsymbol{\Psi} &= \Vi (\BJ - \nabla \varphi)  \quad  \ \rm{in}\ \Omega,\\
\label{dual_proof_p_2}
\divv \boldsymbol{\Psi} &= 0\qquad \rm{in}\ \Omega,\\
\curl \boldsymbol{\Psi} \times \Bn - \Vi \kappa \boldsymbol{\Psi}^t &= 0 \qquad \rm{on}\ \partial\Omega.
\end{align}
By (\ref{dual_proof_p_1}), we have 
\begin{align*}
(\curl \curl \boldsymbol{\Psi} - \kappa^2 \boldsymbol{\Psi}, \nabla q)_{\Omega} 
= (\Vi (\BJ - \nabla \varphi), \nabla q)_{\Omega}\qquad \forall \,q\in C_{0}^{\infty}(\Omega).
\end{align*}
By doing integration by parts and (\ref{dual_proof_p_2}), we have 
\begin{align*}
(\BJ-\nabla \varphi , \nabla q)_{\Omega} = 0\qquad \forall \,q\in C_{0}^{\infty}(\Omega).
\end{align*}
By density argument, we have 
\begin{align*}
(\BJ - \nabla \varphi , \nabla q)_{\Omega} = 0 \qquad \forall \, q \in H^1_0(\Omega).
\end{align*}
We easily obtain $\divv (\BJ-\nabla \varphi) =0$ and $\|\nabla \varphi\|_{0,\Omega} \leq \|\BJ\|_{0,\Omega}$. 
So, we can conclude that the estimate (\ref{est_dual_1}) holds true when $\Omega$ is a bounded, uniformly star-shaped polyhedron (cf. \cite{Hiptmair11}). 

Now we are ready to outline the main results in the following by showing the stability estimates of the discrete solutions from the HDG method (\ref{discrete_mixed_form}) and the associated error estimates.

\begin{theorem}\label{stability_thm}
Let $(\Bw,\Bu,\sigma)$ and $(\Bw_h,\Bu_h,\widehat{\Bu}^t_h,\sigma_h, \widehat{\sigma}_h) $ solve the equations (\ref{pde_mixed_first_order}) and (\ref{discrete_mixed_form}). We assume that (\ref{est_dual_1}) holds with  $0<\alpha \leq \widetilde{\alpha}< \frac{1}{2}$ and $(\Bw,\Bu) \in \BH^{\frac{1}{2}+\alpha}(\Omega)\times \BH^{\frac{1}{2}+\alpha}(\Omega)$.
Then the HDG method (\ref{discrete_mixed_form}) is absolutely stable. When $\kappa h / p \geq C_0$, we have
\begin{align}
\label{stability_u_h1}
\| \Bu_h \|_{\Ct_h} &\leq C\Big(  C^2_{\rm stab}\| 
\frac{\Bf}{\kappa} \|_{0,\Omega}   + C_{\rm stab}
\|\frac{\Bg}{\kappa}\|_{0,\partial \Omega}   \Big) ,\\
\label{stability_w_h1}
\| \Bw_h \|_{\Ct_h} &\leq C\Big(  (\frac{1}{\kappa}+C^2_{\rm stab}) \| \Bf \|_{0,\Omega}   +(\frac{1}{\kappa^{\frac{1}{2}}}+C_{\rm stab})  \|\Bg\|_{0,\partial \Omega}   \Big), \\
\label{stability_ut_h1}
\| \widehat{ \Bu}^t_h \|_{\partial \Ct_h}& \leq C\big( (\frac{\kappa h}{p})^{\frac{1}{2}}
+ ph^{-\frac{1}{2}} \big)\Big( C^2_{\rm stab}\| 
\frac{\Bf}{\kappa} \|_{0,\Omega}   +C_{\rm stab}
\|\frac{\Bg}{\kappa}\|_{0,\partial \Omega}   \Big),
\end{align}
where
$
C_{\rm stab} := 1 + { \frac{\kappa (1+\kappa) h^{2}  }{p^{2}  }} +{ \frac{\kappa (1+\kappa)^2 h^{2\alpha+1}  }{p^{2\alpha+1}  }}+ \frac{(1+\kappa)^2 h^{2\alpha}}{p^{2\alpha}}.
$
\end{theorem}

\begin{theorem}\label{error_thm}
Let $(\Bw,\Bu,\sigma)$ and $(\Bw_h,\Bu_h,\widehat{\Bu}^t_h,\sigma_h, \widehat{\sigma}_h) $ solve the equations (\ref{pde_mixed_first_order}) and (\ref{discrete_mixed_form}). We assume that (\ref{est_dual_1}) holds with  $0<\alpha \leq \widetilde{\alpha}< \frac{1}{2}$ and $(\Bw,\Bu) \in \BH^{\frac{1}{2}+\alpha}(\Omega)\times \BH^{\frac{1}{2}+\alpha}(\Omega)$.
When $\kappa h / p \geq C_0$, we have
\begin{align}
\label{error_ut_h1}
\|\Bu-\Bu_h\|_{\Ct_h} &\leq C \big( R_{\Bw}  \| \Bw \|_{\frac{1}{2}+\alpha,\Omega}+  R_{\Bu} \| \Bu \|_{\frac{1}{2}+\alpha,\Omega}\big),\\
\label{error_wt_h1}
\|\Bw-\Bw_h\|_{\Ct_h}& \leq C\big( (\frac{h}{p})^{\frac{1}{2}+\alpha}  + \kappa R_{\Bw}   \big)  \| \Bw \|_{\frac{1}{2}+\alpha,\Omega} \nn \\& \quad + 
C\big(  \kappa R_{\Bu} +  \kappa( 1+ (1+\kappa)^{-\frac{1}{2}})(\frac{ h}{p})^{\frac{1}{2}+\alpha} \big) \| \Bu \|_{\frac{1}{2}+\alpha,\Omega}  ,
\end{align}
where
$
 R_{\Bw} :=   \frac{(1+\kappa) h^{2\alpha+1}}{p^{2\alpha+1}} +  \frac{(1+\kappa)^{\frac{1}{2}} h^{\alpha+\frac{3}{2}}}{p^{\alpha+\frac{3}{2}}} 
$
and
$  R_{\Bu} := (1+(1+\kappa)^{\frac{1}{2}})  (\frac{ h}{p})^{\frac{1}{2}+\alpha} + (1+\kappa + (1+\kappa)^{\frac{1}{2}})  (\frac{ h}{p})^{2\alpha} .
$
\end{theorem}

\begin{remark}
For the solutions of the first-order system (\ref{pde_mixed_first_order}) which admit the regularity as in (\ref{sta_weak_new}), when $\kappa h/p\leq C_0$, one may tune the parameters $\tau_t$ and $\tau_n$ (cf. Remark \ref{remark_sta_1}) and also get the stability estimates and error estimates for the discrete solutions of the HDG method (\ref{discrete_mixed_form}). When we consider only $\kappa$- and $h$-dependence, the above results hold when  $\Ct_h$ consists of general polyhedral elements. 
\end{remark}

\section{Stability estimate}\label{stability_estimate}
In this section we shall show that the HDG method (\ref{discrete_mixed_form}) is absolutely stable. We first present a lemma which shall be used to estimate the stability estimate of $\Bu_h$.
\begin{lemma}
\label{stability_lemma}
Let $(\Bw_h,\Bu_h,\widehat{\Bu}^t_h,\sigma_h, \widehat{\sigma}_h) $ be the solution of the problem (\ref{discrete_mixed_form}). It holds that
\begin{align}
\label{estimate_mixed_1}
&\|\tau_t^{\frac{1}{2}}( \Bu^t_h - \widehat{\Bu}^t_h ) \|^2_{\partial \Ct_h} + \|\tau_n^{\frac{1}{2}}( \sigma_h - \widehat{\sigma}_h ) \|^2_{\partial \Ct_h}  + \frac{\kappa}{2} \|\widehat{\Bu}^t_h\|^2_{0,\partial \Omega} \leq \|\Bf\|_{0,\Omega} \|\Bu_h\|_{\Ct_h} + \frac{1}{2\kappa} \|\Bg\|^2_{0,\partial \Omega},\\
\label{estimate_mixed_2}
&\|\Bw_h\|^2_{\Ct_h}  \leq \kappa^2 \|\Bu_h\|^2_{\Ct_h} + 2\|\Bf\|_{0,\Omega} \|\Bu_h\|_{\Ct_h}  + \frac{1}{\kappa} \|\Bg\|^2_{0,\partial \Omega}.
\end{align}
\end{lemma}
\begin{proof}
We first choose $\Br_h = \Bw_h, \Bv_h = \Bu_h, \boldsymbol{\eta}_h = \widehat{\Bu}^t_h, q_h=\sigma_h,\xi_h = \widehat{\sigma}_h$ in (\ref{discrete_mixed_form_a})-(\ref{discrete_mixed_form_f}) to get the following equalities:\begin{subequations}
\begin{align}
&(\Vi \Bw_h, \Bw_h )_{\Ct_h} - (\Bu_h,\curl \Bw_h)_{\Ct_h}  + \langle \widehat{\Bu}^t_h\times \Bn,\Bw_h\rangle_{\partial \Ct_h} = 0, \label{pre_sta_pr_1}\\
& (\curl \Bw_h, \Bu_h)_{\Ct_h} + \langle \Bw^t_h \times \Bn, \Bu_h\rangle_{\partial \Ct_h} - \langle  \widehat{\Bw}^t_h \times \Bn, \Bu_h\rangle_{\partial \Ct_h} \nn\\& \qquad \qquad \qquad  \quad + (\Vi \kappa^2 \Bu_h,\Bu_h)_{\Ct_h} - (\sigma_h, \divv \Bu_h)_{\Ct_h} + \langle \widehat{\sigma}_h,\Bu^n_h \cdot \Bn \rangle_{\partial \Ct_h} = (\Bf, \Bu_h)_{\Ct_h},\label{pre_sta_pr_2}\\
& (\divv \Bu_h , \sigma_h)_{\Ct_h} - \langle \Bu^n_h \cdot \Bn ,\sigma_h\rangle_{\partial \Ct_h}  + \langle \widehat{\Bu}^n_h \cdot \Bn,\sigma_h \rangle_{\partial \Ct_h} =0,\label{pre_sta_pr_3}\\
&\langle \widehat{\Bw}^t_h \times \Bn, \widehat{\Bu}^t_h \rangle_{\partial \Ct_h \setminus \partial \Omega}=0,\label{pre_sta_pr_4}\\
& \langle \widehat{\Bw}^t_h \times \Bn, \widehat{\Bu}^t_h \rangle_{\partial \Omega} + \langle \kappa \widehat{\Bu}^t_h, \widehat{\Bu}^t_h \rangle_{\partial \Omega} = \langle \Bg,  \widehat{\Bu}^t_h \rangle_{\partial \Omega},\label{pre_sta_pr_5}\\
& \langle \widehat{\Bu}^n_h \cdot \Bn, \widehat{\sigma}_h \rangle_{\partial \Ct_h} = 0,\label{pre_sta_pr_6}
\end{align}
\end{subequations}
where (\ref{pre_sta_pr_2}) and (\ref{pre_sta_pr_3}) are obtained by integration by parts. Furthermore, noting the definitions of $\widehat{\Bw}_h$ in (\ref{hat_definition}) and applying complex conjugation to (\ref{pre_sta_pr_1}), (\ref{pre_sta_pr_3}) and (\ref{pre_sta_pr_6}), we get the following equalities after simple manipulations:
\begin{align*}
&-(\Vi \Bw_h, \Bw_h )_{\Ct_h} - (\curl \Bw_h, \Bu_h)_{\Ct_h} - \langle \tau_t (\Bu^t_h-\widehat{\Bu}^t_h)^t, \widehat{\Bu}^t_h\rangle_{\partial \Ct_h} + \langle \kappa \widehat{\Bu}^t_h,\widehat{\Bu}^t_h\rangle_{\partial \Omega} = \langle \Bg, \widehat{\Bu}^t_h\rangle_{\partial \Omega},\\
& (\sigma_h, \divv \Bu_h )_{\Ct_h} - \langle \sigma_h,\Bu^n_h \cdot \Bn \rangle_{\partial \Ct_h}  + \langle \sigma_h,\widehat{\Bu}^n_h \cdot \Bn \rangle_{\partial \Ct_h} =0,\\
& -  \langle \widehat{\sigma}_h, \widehat{\Bu}^n_h \cdot \Bn \rangle_{\partial \Ct_h} = 0.
\end{align*}
Adding the above three equalities and (\ref{pre_sta_pr_2}) together and noting that $ (\Bu^t_h-\widehat{\Bu}^t_h)^t =\Bu^t_h-\widehat{\Bu}^t_h $, we have
\begin{align*}
&-(\Vi \Bw_h, \Bw_h )_{\Ct_h} + \langle \tau_t (\Bu^t_h-\widehat{\Bu}^t_h), \Bu^t_h-\widehat{\Bu}^t_h\rangle_{\partial \Ct_h} + \langle \kappa \widehat{\Bu}^t_h,\widehat{\Bu}^t_h\rangle_{\partial \Omega}\\
&\qquad\qquad\qquad \ +  \langle \tau_n (\sigma_h-\widehat{\sigma}_h), \sigma_h-\widehat{\sigma}_h\rangle_{\partial \Ct_h}  + (\Vi \kappa^2 \Bu_h,\Bu_h)_{\Ct_h}   = (\Bf, \Bu_h)_{\Ct_h}+ \langle \Bg, \widehat{\Bu}^t_h\rangle_{\partial \Omega},
\end{align*}
which implies the lemma by the Cauchy-Schwarz inequality.
\end{proof}

Next we shall utilize the dual argument to give the $L^2$-norm estimate of $\Bu_h$. Given $\Bu_h \in \BL^2(\Omega)$, we introduce the first-order system of the dual problem (\ref{dual_problem}) with $\BJ = \Bu_h$:
\begin{subequations}
\label{dual_FOS}
\begin{align}
\Vi \boldsymbol{\Phi} - \curl \boldsymbol{\Psi} &= 0\qquad \rm{in}\ \Omega, \\
\curl \boldsymbol{\Phi} + \Vi\kappa^2 \boldsymbol{\Psi} + \nabla \varphi &=  \Bu_h \quad \ \rm{in}\ \Omega,\\
\label{dual_FOS_1}
\divv \boldsymbol{\Psi} &= 0\qquad \rm{in}\ \Omega,\\
\label{bc1_dual_FOS}
\boldsymbol{\Phi} \times \Bn -  \kappa \boldsymbol{\Psi}^t &= 0 \qquad \rm{on}\ \partial\Omega,\\
\varphi &= 0 \qquad \rm{on}\ \partial\Omega.
\end{align}
\end{subequations}
Due to $\varphi \in H^1_0(\Omega)$, we easily obtain
\begin{align}
\label{est_dual_3_p}
\|\varphi\|_{1,\Omega} \leq C \|\Bu_h\|_{\Ct_h}.
\end{align}
When the estimate (\ref{est_dual_1}) holds, taking $\BJ = \Bu_h$ in (\ref{est_dual_1}) we have
\begin{align}
\label{est_dual_3_new}
\kappa \| \boldsymbol{\Psi} \|_{\frac{1}{2}+\alpha,\Omega} + \|\curl \boldsymbol{\Psi}\|_{\frac{1}{2}+\alpha,\Omega} + \kappa(1+\kappa) \| \boldsymbol{\Psi}\|_{0,\Omega} \leq C(1+\kappa) \|\Bu_h\|_{\Ct_h},
\end{align}
which implies
\begin{align}
\label{est_dual_3_1}
\| \boldsymbol{\Phi}\|_{\frac{1}{2}+\alpha,\Omega}  \leq C (1+\kappa) \|\Bu_h\|_{\Ct_h}.
\end{align}
By the equation (\ref{dual_p_1}) with $\BJ = \Bu_h$, we directly have
\[
(\curl \curl \boldsymbol{\Psi},\boldsymbol{\Psi})_{\Omega} - \kappa^2 (\boldsymbol{\Psi},\boldsymbol{\Psi})_{\Omega} + (\Vi \nabla \varphi, \boldsymbol{\Psi})_{\Omega} = (\Vi \Bu_h, \boldsymbol{\Psi})_{\Omega},
\]
which together with the fact $( \nabla \varphi, \boldsymbol{\Psi})_{\Omega}=0$ and the  boundary condition (\ref{bc_dual_3}) yields
\[
(\curl \boldsymbol{\Psi}, \curl \boldsymbol{\Psi})_{\Omega} - \kappa^2 (\boldsymbol{\Psi},\boldsymbol{\Psi})_{\Omega} - \langle \Vi \kappa \boldsymbol{\Psi}^t, \boldsymbol{\Psi}^t \rangle_{\partial \Omega} = (\Vi \Bu_h, \boldsymbol{\Psi})_{\Omega}.
\]
Thus, taking the imaginary part of the left-hand side of the above equation and using (\ref{est_dual_3_new}) we have
\begin{align}
\label{est_dual_4}
\kappa \|\boldsymbol{\Psi}^t\|^2_{0, \partial \Omega} \leq \|\Bu_h\|_{\Ct_h} \|\boldsymbol{\Psi}\|_{0,\Omega} \leq C \kappa^{-1} \|\Bu_h\|^2_{\Ct_h}.
\end{align}

Next we present a key equality.

\begin{lemma}
\label{lemma_equality_sta}
Let $(\boldsymbol{\Phi},\boldsymbol{\Psi},\varphi)$ be the solution of  the dual problem (\ref{dual_FOS}). We have
\[
\|\Bu_h\|^2_{\Ct_h}  = \sum_{k=1}^6 T_k,
\]
where 
\begin{align*}
T_1 &= \langle \Bu^t_h \times \Bn - \widehat{\Bu}^t_h \times \Bn , \boldsymbol{\Phi} -\boldsymbol{\Pi_V \Phi} \rangle_{\partial \Ct_h},\\
T_2& =  \langle \Bu^n_h \cdot \Bn - \widehat{\Bu}^n_h \cdot \Bn , \varphi - \Pi_Q \varphi \rangle_{\partial \Ct_h},\\
T_3 &= - \langle\tau_t (\Bu^t_h-\widehat{\Bu}^t_h), \boldsymbol{\Psi} - \boldsymbol{\Pi_U\Psi}\rangle_{\partial \Ct_h},\\
T_4 &=- \langle \kappa \widehat{\Bu}^t_h+\widehat{\Bw}^t_h \times \Bn, \boldsymbol{\Psi}^t  \rangle_{\partial \Omega} ,\\
T_5 &= - (\Bf, \boldsymbol{\Pi_U\Psi})_{ \Ct_h},\\
T_6 &=  (\sigma_h, \divv(  \boldsymbol{\Psi} -\boldsymbol{\Pi_U\Psi} ))_{ \Ct_h}  - \langle \widehat{\sigma}_h, (  \boldsymbol{\Psi} -\boldsymbol{\Pi_U\Psi})\cdot \Bn \rangle_{\partial \Ct_h}.
\end{align*}
\end{lemma}
\begin{proof}
Using the dual first-order system (\ref{dual_FOS}), we obtain
\begin{align}
\|\Bu_h\|^2_{\Ct_h} &= (\Bu_h, \curl \boldsymbol{\Phi} + \Vi\kappa^2 \boldsymbol{\Psi} + \nabla \varphi )_{\Ct_h} + (\Bw_h,\Vi \boldsymbol{\Phi} - \curl \boldsymbol{\Psi})_{\Ct_h}\nn\\
& = (\Bu_h, \curl \boldsymbol{\Pi_V \Phi} + \Vi\kappa^2 \boldsymbol{\Pi_U \Psi} + \nabla \Pi_Q \varphi)_{\Ct_h}  + (\Bw_h,\Vi \boldsymbol{\Pi_V \Phi} - \curl \boldsymbol{\Pi_U \Psi})_{\Ct_h}\nn\\
&\quad \ + (\Bu_h,\curl (\boldsymbol{\Phi} -\boldsymbol{\Pi_V \Phi}   ) )_{\Ct_h} + (\Bu_h,\Vi\kappa^2(  \boldsymbol{\Psi}  - \boldsymbol{\Pi_U\Psi} ))_{\Ct_h} + (\Bu_h, \nabla(\varphi- \Pi_Q \varphi))_{\Ct_h}\nn\\
&\quad \ + (\Bw_h,\Vi(\boldsymbol{\Phi} -\boldsymbol{\Pi_V\Phi} ))_{\Ct_h} - (\Bw_h, \curl (\boldsymbol{\Psi}-\boldsymbol{\Pi_U\Psi}  ))_{\Ct_h}. \label{proof_sta_1}
\end{align}
By the definitions of $\boldsymbol \Pi_U$ and $\boldsymbol \Pi_V$, we have $(\Bu_h,\Vi\kappa^2(  \boldsymbol{\Psi}  - \boldsymbol{\Pi_U\Psi} ))_{\Ct_h} =0$ and $(\Bw_h,\Vi(\boldsymbol{\Phi} -\boldsymbol{\Pi_V\Phi} ))_{\Ct_h}=0$. Integrating by parts and applying the property of the $L^2$-projections yields
\begin{align}
(\Bu_h,\curl (\boldsymbol{\Phi} -\boldsymbol{\Pi_V \Phi}   ) )_{\Ct_h} &= (\curl \Bu_h, \boldsymbol{\Phi} -\boldsymbol{\Pi_V \Phi})_{\Ct_h} + \langle \Bu_h \times \Bn, \boldsymbol{\Phi} -\boldsymbol{\Pi_V \Phi} \rangle_{\partial \Ct_h} \nn\\
&= \langle \Bu_h \times \Bn, \boldsymbol{\Phi} -\boldsymbol{\Pi_V \Phi} \rangle_{\partial \Ct_h}\nn\\
&= \langle \Bu^t_h \times \Bn, \boldsymbol{\Phi} -\boldsymbol{\Pi_V \Phi} \rangle_{\partial \Ct_h},\label{sta_equ_pre_1}
\end{align}
\begin{align}
 (\Bu_h, \nabla(\varphi - \Pi_Q \varphi))_{\Ct_h} &= -(\divv \Bu_h,\varphi-\Pi_Q \varphi)_{\Ct_h}+ \langle \Bu^n_h \cdot \Bn, \varphi-\Pi_Q\varphi\rangle_{\partial \Ct_h}  \nn\\
 & = \langle \Bu^n_h \cdot \Bn, \varphi-\Pi_Q\varphi\rangle_{\partial \Ct_h} ,\label{sta_equ_pre_2}
\end{align}
and
\begin{align}
- (\Bw_h, \curl (\boldsymbol{\Psi}-\boldsymbol{\Pi_U\Psi}  ))_{\Ct_h} &= -( \curl \Bw_h,  \boldsymbol{\Psi}-\boldsymbol{\Pi_U\Psi} )_{\Ct_h} - \langle \Bw_h \times \Bn, \boldsymbol{\Psi}-\boldsymbol{\Pi_U\Psi}\rangle_{\partial \Ct_h}\nn \\
& = - \langle \Bw^t_h \times \Bn, \boldsymbol{\Psi}^t \rangle_{\partial \Ct_h}  + \langle \Bw^t_h \times \Bn,\boldsymbol{\Pi_U\Psi}\rangle_{\partial \Ct_h}.\label{sta_equ_pre_3}
\end{align}
Taking $\Br_h =  \boldsymbol{\Pi_V \Phi}$ in the equation (\ref{discrete_mixed_form_a}), noting that $\widehat{\Bu}^t_h\times \Bn$ is continuous across each interior face and using the boundary condition (\ref{bc1_dual_FOS}), we obtain
\begin{align}
(\Bu_h, \curl \boldsymbol{\Pi_V \Phi} )_{\Ct_h} &= (\Vi \Bw_h,\boldsymbol{\Pi_V \Phi} )_{\Ct_h} + \langle \widehat{\Bu}^t_h\times \Bn, \boldsymbol{\Pi_V \Phi}\rangle_{\partial \Ct_h} \nn\\
& =  (\Vi \Bw_h,\boldsymbol{\Pi_V \Phi} )_{\Ct_h} + \langle \widehat{\Bu}^t_h\times \Bn, \boldsymbol{\Pi_V \Phi} -  \boldsymbol{ \Phi} \rangle_{\partial \Ct_h} - \langle \widehat{\Bu}^t_h, \boldsymbol{ \Phi}\times\Bn\rangle_{\partial \Omega}\nn\\
& =  (\Vi \Bw_h,\boldsymbol{\Pi_V \Phi} )_{\Ct_h} + \langle \widehat{\Bu}^t_h\times \Bn, \boldsymbol{\Pi_V \Phi} -  \boldsymbol{ \Phi} \rangle_{\partial \Ct_h} - \langle \widehat{\Bu}^t_h, \kappa \boldsymbol{\Psi}^t \rangle_{\partial \Omega}.\label{sta_equ_pre_4}
\end{align}
Taking $q_h = \Pi_Q \varphi$ in (\ref{discrete_mixed_form_c}) and noting that $\widehat{\Bu}^n_h \cdot \Bn$ is continuous across each interior face and $\varphi \in H^1_0(\Omega)$ we have
\begin{align}
(\Bu_h,\nabla \Pi_Q \varphi )_{\Ct_h} = \langle \widehat{\Bu}^n_h \cdot \Bn, \Pi_Q \varphi \rangle_{\partial \Ct_h}  = \langle \widehat{\Bu}^n_h \cdot \Bn, \Pi_Q \varphi - \varphi \rangle_{\partial \Ct_h}  .\label{sta_equ_pre_5}
\end{align}
We further take $\Bv_h = \boldsymbol{\Pi_U \Psi} $ in (\ref{discrete_mixed_form_b}) to get
\begin{align}
&-( \Bw_h,\curl  \boldsymbol{\Pi_U\Psi})_{\Ct_h} \nn\\
&= - \langle \widehat{\Bw}^t_h \times \Bn,\boldsymbol{\Pi_U\Psi}\rangle_{\partial \Ct_h} + ( \Vi \kappa^2 \Bu_h,\boldsymbol{\Pi_U\Psi})_{\Ct_h} - ( \sigma_h,\divv \boldsymbol{\Pi_U\Psi} )_{\Ct_h} \nn\\
&\quad + \langle \widehat{\sigma}_h,\boldsymbol{\Pi_U\Psi} \cdot \Bn \rangle_{\partial \Ct_h} - (\Bf, \boldsymbol{\Pi_U\Psi})_{\Ct_h}\nn\\
&= - \langle {\Bw}^t_h \times \Bn - \tau_t (\Bu^t_h-\widehat{\Bu}^t_h),\boldsymbol{\Pi_U\Psi} \rangle_{\partial \Ct_h} + ( \Vi \kappa^2 \Bu_h,\boldsymbol{\Pi_U\Psi} )_{\Ct_h}  - ( \sigma_h,\divv\boldsymbol{\Pi_U\Psi} -\divv \boldsymbol{\Psi} )_{\Ct_h} \nn\\
&\quad + \langle \widehat{\sigma}_h,\boldsymbol{\Pi_U\Psi} \cdot \Bn - \Psi \cdot \Bn \rangle_{\partial \Ct_h} - (\Bf, \boldsymbol{\Pi_U\Psi})_{\Ct_h},\label{sta_equ_pre_6}
\end{align}
where the above second equality holds due to the fact that $\divv \boldsymbol{\Psi}=0$, $\widehat{\sigma}_h$ is continuous across each interior face and $\widehat{\sigma}_h=0$ on $\Ce^{\partial}_h$. Inserting the above equalities (\ref{sta_equ_pre_1})-(\ref{sta_equ_pre_6}) into the right-hand side of (\ref{proof_sta_1}), we obtain the result. This completes the proof.
\end{proof}

We can now give the proof of Theorem \ref{stability_thm}.
\begin{proof} (Proof of Theorem \ref{stability_thm}) 
We derive the upper bounds for $T_1,\cdots,T_6$ in Lemma \ref{lemma_equality_sta} under the assumptions in Theorem \ref{stability_thm}. By the Cauchy-Schwarz inequality, the approximation properties of standard $L^2$-projections, the inequalities (\ref{est_dual_3_p}) and (\ref{est_dual_3_new}), we have
\begin{align*}
T_1& \leq C \| \tau_t^{\frac{1}{2}} (\Bu^t_h- \widehat{\Bu}^t_h)  \|_{\partial \Ct_h} \tau_t^{-\frac{1}{2}} (\frac{h}{p})^{\alpha}\|\boldsymbol{\Phi}\|_{\frac{1}{2}+\alpha,\Omega} \leq  C \| \tau_t^{\frac{1}{2}} (\Bu^t_h- \widehat{\Bu}^t_h)  \|_{\partial \Ct_h} \tau_t^{-\frac{1}{2}} (\frac{h}{p})^{\alpha} (1+\kappa) \|\Bu_h\|_{\Ct_h},\\
T_2 &= - \langle \tau_n(\sigma_h - \widehat{\sigma}_h) , \varphi - \Pi_Q \varphi \rangle_{\partial \Ct_h} \leq C \| \tau^{\frac{1}{2}}_n ( \sigma_h - \widehat{\sigma}_h ) \|_{\partial \Ct_h} \tau_n^{\frac{1}{2}} (\frac{h}{p})^{\frac{1}{2}} \| \varphi \|_{1,\Omega} \\&\leq C \| \tau^{\frac{1}{2}}_n ( \sigma_h - \widehat{\sigma}_h ) \|_{\partial \Ct_h} \tau_n^{\frac{1}{2}} (\frac{h}{p})^{\frac{1}{2}}  \|\Bu_h\|_{\Ct_h},\\
T_3 &\leq C \| \tau_t^{\frac{1}{2}} (\Bu^t_h- \widehat{\Bu}^t_h)  \|_{\partial \Ct_h} \tau_t^{\frac{1}{2}} (\frac{h}{p})^{\alpha}\|\boldsymbol{\Psi}\|_{\frac{1}{2}+\alpha,\Omega} \leq  C \| \tau_t^{\frac{1}{2}} (\Bu^t_h- \widehat{\Bu}^t_h)  \|_{\partial \Ct_h} \tau_t^{\frac{1}{2}} (\frac{h}{p})^{\alpha} (1+1/\kappa) \|\Bu_h\|_{\Ct_h}.
\end{align*}
Taking $\boldsymbol{\eta}_h = \BP_{\BM} \boldsymbol{\Psi}^t$ in (\ref{discrete_mixed_form_e}) and using the property of the $L^2$-projection operator $\BP_{\BM}$ on $\Ce^{\partial}_h$ and the inequality (\ref{est_dual_4}) yields
\[
T_4 = - \langle \kappa \widehat{\Bu}^t_h+\widehat{\Bw}^t_h \times \Bn,  \BP_{\BM}\boldsymbol{\Psi}^t  \rangle_{\partial \Omega}  = -\langle \Bg , \BP_{\BM}\boldsymbol{\Psi}^t  \rangle_{\partial \Omega}  \leq \|\Bg\|_{0, \partial \Omega} \| \boldsymbol{\Psi}^t \|_{0,\partial \Omega}  \leq C \kappa^{-1} \|\Bg\|_{0,\partial \Omega} \|\Bu_h\|_{\Ct_h}.
\]
For the estimate of $T_5$, we easily deduce
\begin{align*}
T_5 \leq  \| \Bf  \|_{0,\Omega}  \| \boldsymbol{\Psi}  \|_{0,\Omega}  \leq C \| \Bf  \|_{0,\Omega} \kappa^{-1}  \|\Bu_h\|_{\Ct_h}.
\end{align*}
Applying integration by parts on $T_6$, we have
\begin{align*}
T_6 &= -(\nabla \sigma_h,    \boldsymbol{\Psi}-\boldsymbol{\Pi_U\Psi} )_{ \Ct_h}  + \langle \sigma_h - \widehat{\sigma}_h, (\boldsymbol{\Psi} -\boldsymbol{\Pi_U\Psi} )\cdot \Bn \rangle_{\partial \Ct_h} = \langle \sigma_h - \widehat{\sigma}_h, (\boldsymbol{\Psi} -\boldsymbol{\Pi_U\Psi} )\cdot \Bn \rangle_{\partial \Ct_h} \\
&\leq C \| \tau^{\frac{1}{2}}_n (\sigma_h - \widehat{\sigma}_h) \|_{\partial \Ct_h} \tau^{-\frac{1}{2}}_n  (\frac{h}{p})^{\alpha} \|\boldsymbol{\Psi}  \|_{\frac{1}{2}+\alpha,\Omega} \leq C \| \tau^{\frac{1}{2}}_n (\sigma_h - \widehat{\sigma}_h) \|_{\partial \Ct_h} \tau^{-\frac{1}{2}}_n  (\frac{h}{p})^{\alpha} (1+1/\kappa)  \|\Bu_h\|_{\Ct_h}.
\end{align*}
Combining the above estimates for $T_1,\cdots, T_6$, we obtain
\begin{align}
\|\Bu_h\|^2_{\Ct_h} & \leq C \kappa^{-1}\left(\| \Bf  \|_{0,\Omega}+  \|\Bg\|_{0,\partial \Omega} \right)\|\Bu_h\|_{\Ct_h}\nn \\
&\quad + C
\big(\tau_t^{-\frac{1}{2}} (\frac{h}{p})^{\alpha}(1+\kappa)  +  \tau_t^{\frac{1}{2}} (\frac{h}{p})^{\alpha}  (1+1/\kappa) \big)\| \tau_t^{\frac{1}{2}} (\Bu^t_h- \widehat{\Bu}^t_h)  \|_{\partial \Ct_h} \|\Bu_h\|_{\Ct_h} \nn\\
& \quad + C \big( \tau_n^{\frac{1}{2}} (\frac{h}{p})^{\frac{1}{2}}+\tau^{-\frac{1}{2}}_n  (\frac{h}{p})^{\alpha} (1+1/\kappa) \big)  \| \tau^{\frac{1}{2}}_n (\sigma_h - \widehat{\sigma}_h) \|_{\partial \Ct_h}    \|\Bu_h\|_{\Ct_h}. \label{uh_estimate_basis}
\end{align}
Here we choose $\tau_t = \frac{p}{h}$ and $\tau_n = \frac{(1+\kappa) h}{p}$. By the Young's inequality, we have
\begin{align}
\|\Bu_h\|^2_{\Ct_h} \leq C \Big(& \kappa^{-2}\| \Bf  \|^2_{0,\Omega}+  \kappa^{-2}\|\Bg\|^2_{0,\partial \Omega} \nn\\
&+ \big((1+\kappa)^2 (\frac{h}{p})^{2\alpha+1}  +(1+1/\kappa)^2(\frac{h}{p})^{2\alpha-1} \big)\| \tau_t^{\frac{1}{2}} (\Bu^t_h- \widehat{\Bu}^t_h)  \|^2_{\partial \Ct_h} \nn\\
&+ \big((1+\kappa) (\frac{h}{p})^2  + \frac{(1+\kappa)}{\kappa^2} (\frac{h}{p})^{2\alpha-1}\big)\| \tau^{\frac{1}{2}}_n (\sigma_h - \widehat{\sigma}_h) \|^2_{\partial \Ct_h}  \Big).\nn
\end{align}
Combining the above estimate and (\ref{estimate_mixed_1}), the absolutely stable property of $\|\Bu_h\|_{\Ct_h}$ can be easily observed, and we can further obtain (\ref{stability_u_h1}) by the Young's inequality in the regime $\kappa h/p \geq C_0$. Then, by the estimate (\ref{estimate_mixed_2}), we can also see the absolutely stable property of $\|\Bw_h\|_{\Ct_h}$ and have
\begin{align}
\|\Bw_h\|^2_{\Ct_h}  \leq2 \kappa^2 \|\Bu_h\|^2_{\Ct_h} + \|\frac{\Bf}{\kappa}\|^2_{0,\Omega}  + \frac{1}{\kappa} \|\Bg\|^2_{0,\partial \Omega}.\nn
\end{align}
Then (\ref{stability_w_h1}) is derived by (\ref{stability_u_h1}). Furthermore, combining the fact that (cf. \cite{Schwab98}) 
\begin{align}
\| \Bu^t_h \|_{\partial \Ct_h}\leq C p h^{-\frac{1}{2}} \|\Bu_h\|_{\Ct_h},\nn
\end{align}
(\ref{estimate_mixed_1}), (\ref{stability_u_h1}) and the triangular inequality yields the absolutely stable property of $\| \widehat{\Bu}^t_h \|_{\partial \Ct_h}$ and the estimate (\ref{stability_ut_h1}).

When $\Bf =0$ and $\Bg=0$ in the first-order system (\ref{pde_mixed_first_order}), the estimates (\ref{stability_u_h1})-(\ref{stability_ut_h1}) and Lemma  \ref{stability_lemma} imply $\Bw_h=0,\Bu_h=0$ on $\Ct_h$ and $\widehat{\Bu}^t_h=0, \sigma_h =   \widehat{\sigma}_h$ on $\partial \Ct_h$. It then follows from (\ref{discrete_mixed_form_b}) that for any $\Bv_h \in \BU_h$,
\[
- ( \sigma_h,\divv \Bv_h )_{\Ct_h} + \langle \widehat{\sigma}_h, \Bv_h \cdot \Bn \rangle_{\partial \Ct_h}= ( \nabla \sigma_h, \Bv_h )_{\Ct_h} -  \langle {\sigma}_h, \Bv_h \cdot \Bn \rangle_{\partial \Ct_h}+\langle \widehat{\sigma}_h, \Bv_h \cdot \Bn \rangle_{\partial \Ct_h} = 0,
\]
which implies $\sigma_h$ is piecewise constant on $\Ct_h$. Due to the fact that $\sigma_h = \widehat{\sigma}_h=0$ on $\partial \Omega$, we have $\sigma_h = 0$ on $\Ct_h$ and $\widehat{\sigma}_h=0$ on $\partial \Ct_h$. Hence, the well-posedness of the HDG method (\ref{discrete_mixed_form}) always holds without imposing any mesh constraint, i.e., the HDG method (\ref{discrete_mixed_form}) is absolutely stable.
\end{proof}

Moreover, under the assumptions made in Theorem \ref{stability_thm}, we can further get the upper bounds for $\|\curl \Bu_h\|_{\Ct_h}$ and $\|\divv \Bu_h  \|_{\Ct_h}$. We take $\Br_h = \curl \Bu_h$ in (\ref{discrete_mixed_form_a}) to get
\[
(\Vi \Bw_h,\curl \Bu_h)_{\Ct_h} - (\Bu_h,\curl \curl \Bu_h)_{\Ct_h} + \langle\widehat{\Bu}^t_h\times \Bn, \curl \Bu_h \rangle_{\partial \Ct_h}  =0 .
\]
Using integration by parts on the above equation, we have
\[
(\Vi \Bw_h,\curl \Bu_h)_{\Ct_h} - ( \curl \Bu_h,\curl \Bu_h )_{\Ct_h} - \langle{\Bu}^t_h\times \Bn, \curl \Bu_h \rangle_{\partial \Ct_h}  + \langle\widehat{\Bu}^t_h\times \Bn, \curl \Bu_h \rangle_{\partial \Ct_h}  =0,
\]
which directly yields
\[
\|  \curl \Bu_h \|^2_{\Ct_h} \leq \| \Bw_h \|_{\Ct_h}\|  \curl \Bu_h \|_{\Ct_h}  + C  \| \tau_t^{\frac{1}{2}} (\Bu^t_h- \widehat{\Bu}^t_h)  \|_{\partial \Ct_h} \tau_t^{-\frac{1}{2}} p h^{-\frac{1}{2}} \|  \curl \Bu_h \|_{\Ct_h} .
\]
Combining the above inequality, (\ref{estimate_mixed_1}), (\ref{stability_u_h1}) and (\ref{stability_w_h1}), then we get
\[
\| \curl \Bu_h \|_{\Ct_h}\leq C \Big( \big(\frac{1}{\kappa} + p^{\frac{1}{2}} +(1+\frac{p^{\frac{1}{2}}}{\kappa})C^2_{\rm stab}  \big)\|\Bf\|_{0,\Omega}+  \big( \frac{1}{\kappa^{\frac{1}{2}}}  + \frac{p^{\frac{1}{2}}}{ \kappa^{\frac{1}{2}} } +(1+\frac{p^{\frac{1}{2}}}{\kappa})C_{\rm stab}    \big)\|\Bg\|_{0,\partial\Omega} \Big).
\]
Taking $q_h = \divv \Bu_h$ in (\ref{discrete_mixed_form_c}) and using integration by parts, we have
\[
\|\divv \Bu_h  \|^2_{\Ct_h} = \langle {\Bu}^n_h\cdot \Bn-\widehat{\Bu}^n_h\cdot \Bn,\divv \Bu_h \rangle_{\partial \Ct_h}\leq C \| \tau^{\frac{1}{2}}_n ( \sigma_h -   \widehat{\sigma}_h) \|_{\partial \Ct_h}  \tau^{\frac{1}{2}}_n p h^{-\frac{1}{2}} \|\divv \Bu_h  \|_{\Ct_h}.
\]
Then we obtain the upper bound for  $\|\divv \Bu_h  \|_{\Ct_h}$ by the above estimate, (\ref{estimate_mixed_1}) and (\ref{stability_u_h1}) as follows:
\[
\| \divv \Bu_h \|_{\Ct_h} \leq C (1+\kappa)^{\frac{1}{2}}p^{\frac{1}{2}}\Big( (1+\frac{C^2_{\rm stab}}{\kappa}) \| \Bf \|_{0,\Omega}   + \big( \frac{1}{\kappa^{\frac{1}{2}}} + \frac{C_{\rm stab}}{\kappa} \big) \|\Bg\|_{0,\partial \Omega}   \Big).
\]

\begin{remark}\label{remark_sta_1}
By the estimates (\ref{estimate_mixed_1}) and (\ref{stability_u_h1}), we can get the upper bound for $\| \tau^{\frac{1}{2}}_n ( \sigma_h -   \widehat{\sigma}_h) \|_{\partial \Ct_h}  $. Moreover, taking $\Bv_h = \nabla \sigma_h$ in (\ref{discrete_mixed_form_b}) and applying integration by parts, the Cauchy-Schwarz inequality, trace inequality and the estimates in Lemma \ref{stability_lemma} and Theorem \ref{stability_thm}, we can also get the stability estimate for $\| \nabla \sigma_h \|_{\Ct_h}$. When $\kappa h/p \leq C_0$, one may tune the parameters $\tau_t$ and $\tau_n$ according to the derivation of upper bound for the right-hand side of (\ref{uh_estimate_basis}) and get the stability estimates.
\end{remark}

\section{Error analysis}\label{error_analysis}
In this section we provide detailed proofs of the a priori error estimates in Theorem \ref{error_thm}. We denote
\begin{align*}
&\Be_{\Bw} = \boldsymbol{\Pi}_{\BV} \Bw-\Bw_h,\quad \Be_{\widehat{\Bw}^t} = \BP_{\BM} \Bw^t-\widehat{\Bw}^t_h,\quad \Be_{\Bu} = \boldsymbol{\Pi}_{\BU} \Bu-\Bu_h,\quad \Be_{\widehat{\Bu}^t} = \BP_{\BM} \Bu^t-\widehat{\Bu}^t_h,\\
&\Be_{\widehat{\Bu}^n} = \BP_{\BM} \Bu^n-\widehat{\Bu}^n_h,\quad e_{\sigma} = \Pi_{Q} \sigma - \sigma_h, \quad e_{\widehat{\sigma}}  = P_M \sigma - \widehat{\sigma}_h.
\end{align*}
In the following we first present the error equation for the analysis.

\begin{lemma}
Let $(\Bw,\Bu,\sigma)$ and $(\Bw_h,\Bu_h,\widehat{\Bu}^t_h,\sigma_h, \widehat{\sigma}_h) $ solve the equations (\ref{pde_mixed_first_order}) and (\ref{discrete_mixed_form}). We have
\begin{subequations}
\label{error_eq}
\begin{align}
\label{error_eq_a}
&(\Vi \Be_{\Bw},\Br_h)_{\Ct_h} - (\Be_{\Bu},\curl \Br_h)_{\Ct_h} + \langle \Be_{\widehat{\Bu}^t} \times \Bn, \Br_h\rangle_{\partial \Ct_h} = 0,\\
\label{error_eq_b}
&(\Be_{\Bw},\curl \Bv_h)_{\Ct_h} -  \langle \Be_{\widehat{\Bw}^t} \times \Bn, \Bv_h\rangle_{\partial \Ct_h} + (\Vi \kappa^2\Be_{\Bu},\Bv_h)_{\Ct_h} \nn\\
&\qquad \qquad \qquad  \ - (e_{\sigma},\divv \Bv_h)_{\Ct_h}+ \langle e_{\widehat{\sigma}},\Bv_h \cdot \Bn \rangle_{\partial \Ct_h}=0,\\
\label{error_eq_c}
&-(\Be_{\Bu},\nabla q_h)_{\Ct_h}+ \langle \Be_{\widehat{\Bu}^n} \cdot \Bn, q_h\rangle_{\partial \Ct_h} = 0,\\
\label{error_eq_d}
&  \langle \Be_{\widehat{\Bw}^t} \times \Bn, \boldsymbol{\eta}_h\rangle_{\partial \Ct_h \setminus \partial \Omega}=0,\\
\label{error_eq_e}
&\langle \Be_{\widehat{\Bw}^t} \times \Bn, \boldsymbol{\eta}_h\rangle_{ \partial \Omega}+ \langle \kappa \Be_{\widehat{\Bu}^t}, \boldsymbol{\eta}_h\rangle_{ \partial \Omega}=0,\\
\label{error_eq_f}
&\langle \Be_{\widehat{\Bu}^n} \cdot \Bn,{\xi}_h \rangle_{\partial \Ct_h}=0,
\end{align}
\end{subequations}
for all $(\Br_h, \Bv_h, \boldsymbol{\eta}_h, q_h,\xi_h) \in \BV_h \times \BU_h \times \BM^t_h \times  Q_h \times M_h(0)$.
\end{lemma}
\begin{proof}
We notice that the exact solution $(\Bw,\Bu,\Bu^t|_{\Ce_h},\sigma,\sigma|_{\Ce_h})$ also satisfies the equation (\ref{discrete_mixed_form}). Hence, due to the property of standard $L^2$-projection, the solutions $\Bw_h$, $\widehat{\Bw}^t_h|_{\Ce_h}$, $\Bu_h$, $\widehat{\Bu}^t_h|_{\Ce_h}$, $\widehat{\Bu}^n_h|_{\Ce_h}$, $\sigma_h$, $\widehat{\sigma}_h|_{\Ce_h} $ in the equation (\ref{discrete_mixed_form}) can be replaced by $\boldsymbol{\Pi_V}\Bw$, $\boldsymbol{P_M}\Bw^t|_{\Ce_h}$, $\boldsymbol{\Pi_u}\Bu$, $\boldsymbol{P_M}\Bu^t|_{\Ce_h}$, $\boldsymbol{P_M}\Bu^n|_{\Ce_h}$, $\Pi_Q\sigma, P_M\sigma|_{\Ce_h} $ respectively to derive a new equation, which subtracts the equation (\ref{discrete_mixed_form}) to yield the result.
\end{proof}

Next we are going to present our first error estimate.
\begin{lemma}
\label{error_analysis_1}
If we choose $\tau_t = \frac{p}{h}$ and $\tau_n = \frac{(1+\kappa) h}{p}$, we have
\begin{align}
& \|\kappa^{\frac{1}{2}}  \Be_{\widehat{\Bu}^t} \|_{0, \partial \Omega } + \|  \tau_t^{\frac{1}{2}}(\Be_{\Bu}^t- \Be_{\widehat{\Bu}^t} )   \|_{\partial \Ct_h}  + \|\tau_n^{\frac{1}{2}}  ( e_\sigma- e_{\widehat{\sigma}} ) \|_{\partial \Ct_h} \leq C \eta(\Bw,\Bu)  , \label{pre_error_analysis_1} \\
&\| \Be_{\Bw} \|_{\Ct_h} \leq \kappa \| \Be_{\Bu} \|_{\Ct_h}  + C\eta(\Bw,\Bu),\label{pre_error_analysis_2} 
\end{align}
where $\eta(\Bw,\Bu) =  \frac{h^t}{p^t}\|\Bw\|_{t,\Omega} + \big(1+(1+\kappa) ^{-\frac{1}{2}}\big) \frac{h^{s-1}}{p^{s-1}} \|\Bu\|_{s,\Omega} $, $s>\frac{1}{2},t>\frac{1}{2}$.
\end{lemma}
\begin{proof}
Let $\Br_h = \Be_{\Bw}, \Bv_h = \Be_{\Bu}, \boldsymbol{\eta}_h = \Be_{\widehat{\Bu}^t}, q_h=e_{\sigma}, \xi_h = e_{\widehat{\sigma}}$ in the error equation (\ref{error_eq}). Then we get the following equalities after some simple manipulations which includes applying integration by parts:
\begin{subequations}
\label{error_eq1}
\begin{align}
\label{error_eq1_a}
&-(\Vi \Be_{\Bw},\Be_{\Bw})_{\Ct_h} - (\Be_{\Bw},\curl \Be_{\Bu})_{\Ct_h} + \langle \Be^t_{\Bw}\times \Bn, \Be^t_{\Bu}\rangle_{\partial \Ct_h} - \langle \Be^t_{\Bw}\times \Bn, \Be_{\widehat{\Bu}^t}\rangle_{\partial \Ct_h} = 0,\\
&(\Be_{\Bw},\curl \Be_{\Bu})_{\Ct_h} -  \langle \Be_{\widehat{\Bw}^t} \times \Bn, \Be^t_{\Bu}\rangle_{\partial \Ct_h} + (\Vi \kappa^2\Be_{\Bu},\Be_{\Bu})_{\Ct_h} - (e_{\sigma},\divv \Be_{\Bu})_{\Ct_h}+ \langle e_{\widehat{\sigma}},\Be^n_{\Bu} \cdot \Bn \rangle_{\partial \Ct_h}=0,\\
&(e_\sigma,\divv \Be_{\Bu})_{\Ct_h} - \langle e_\sigma, \Be^n_{\Bu} \cdot \Bn\rangle_{\partial \Ct_h} + \langle e_\sigma, \Be_{\widehat{\Bu}^n} \cdot \Bn\rangle_{\partial \Ct_h} = 0,\\
&  \langle \Be_{\widehat{\Bw}^t} \times \Bn, \Be_{\widehat{\Bu}^t}\rangle_{\partial \Ct_h \setminus \partial \Omega}=0,\\
\label{error_eq1_e}
&\langle \Be_{\widehat{\Bw}^t} \times \Bn, \Be_{\widehat{\Bu}^t}\rangle_{ \partial \Omega}+ \langle \kappa \Be_{\widehat{\Bu}^t} , \Be_{\widehat{\Bu}^t}\rangle_{ \partial \Omega}=0,\\
\label{error_eq1_f}
&-\langle e_{\widehat{\sigma}}, \Be_{\widehat{\Bu}^n} \cdot \Bn \rangle_{\partial \Ct_h}=0.
\end{align}
\end{subequations}
Adding the above equalities (\ref{error_eq1_a})-(\ref{error_eq1_f}) together yields
\begin{align}
-(\Vi \Be_{\Bw},\Be_{\Bw})_{\Ct_h}&+(\Vi \kappa^2\Be_{\Bu},\Be_{\Bu})_{\Ct_h} + \langle (\Be^t_{\Bw}-\Be_{\widehat{\Bw}^t})\times \Bn, \Be^t_{\Bu}-\Be_{\widehat{\Bu}^t}\rangle_{\partial \Ct_h}\nn \\
&+\langle \kappa \Be_{\widehat{\Bu}^t} , \Be_{\widehat{\Bu}^t}\rangle_{ \partial \Omega} - \langle e_\sigma- e_{\widehat{\sigma}},(\Be^n_{\Bu} - \Be_{\widehat{\Bu}^n}) \cdot \Bn \rangle_{\partial \Ct_h}=0. \label{error_eq2}
\end{align}
By the definition of $\widehat{\Bw}_h$ in (\ref{hat_definition}) we have
\begin{align}
(\Be^t_{\Bw}-\Be_{\widehat{\Bw}^t})\times \Bn & = (\boldsymbol{\Pi}_{\BV} \Bw-\Bw_h)^t \times \Bn - (\BP_{\BM} \Bw^t-\widehat{\Bw}^t_h)\times \Bn \nn  \\
& = (\boldsymbol{\Pi}_{\BV} \Bw)^t\times \Bn-\BP_{\BM} \Bw^t\times \Bn - \tau_t (\Bu^t_h-\widehat{\Bu}^t_h)\label{error_eq3} \\
& = (\boldsymbol{\Pi}_{\BV} \Bw)^t\times \Bn-\BP_{\BM} \Bw^t\times \Bn - \tau_t \big( (\boldsymbol{\Pi}_{\BU} \Bu  - \Be_{\Bu})^t  - (\BP_{\BM} \Bu^t- \Be_{\widehat{\Bu}^t} ) \big)\nn \\
& =  (\boldsymbol{\Pi}_{\BV} \Bw)^t\times \Bn-\BP_{\BM} \Bw^t\times \Bn +\tau_t(\Be_{\Bu}^t- \Be_{\widehat{\Bu}^t} )-\tau_t \big( (\boldsymbol{\Pi}_{\BU} \Bu)^t - \BP_{\BM} \Bu^t  \big).\nn
\end{align}
Moreover, by the definition of $\widehat{\Bu}^n_h$ in (\ref{hat_definition}), we have
\begin{align}
(\Be^n_{\Bu} - \Be_{\widehat{\Bu}^n}) \cdot \Bn & = ( \boldsymbol{\Pi}_{\BU} \Bu -\Bu_h  )\cdot\Bn - (\BP_{\BM} \Bu^n-\widehat{\Bu}^n_h)\cdot\Bn \nn\\
& = (  \boldsymbol{\Pi}_{\BU} \Bu- \BP_{\BM} \Bu^n )\cdot\Bn+\tau_n(\sigma_h-\widehat{\sigma}_h)\nn\\
& = (  \boldsymbol{\Pi}_{\BU} \Bu- \BP_{\BM} \Bu^n )\cdot\Bn+ \tau_n \big( ( \Pi_Q \sigma - e_\sigma ) -( P_M \sigma - e_{\widehat{\sigma}} )   \big)\nn\\
& = (  \boldsymbol{\Pi}_{\BU} \Bu- \BP_{\BM} \Bu^n )\cdot\Bn- \tau_n( e_\sigma- e_{\widehat{\sigma}})+\tau_n (  \Pi_Q \sigma - P_M \sigma).\label{error_eq4}
\end{align}
Inserting (\ref{error_eq3}) and (\ref{error_eq4}) into (\ref{error_eq2}), we obtain
\begin{align}
&-\Vi \| \Be_{\Bw} \|^2_{\Ct_h}+\Vi \kappa^2\|\Be_{\Bu}\|^2_{\Ct_h} +\kappa \| \Be_{\widehat{\Bu}^t} \|^2_{0, \partial \Omega } + \|  \tau_t^{\frac{1}{2}}(\Be_{\Bu}^t- \Be_{\widehat{\Bu}^t} )   \|^2_{\partial \Ct_h}  + \|\tau_n^{\frac{1}{2}}  ( e_\sigma- e_{\widehat{\sigma}} ) \|^2_{\partial \Ct_h} \nn\\
&= - \langle (\boldsymbol{\Pi}_{\BV} \Bw)^t\times \Bn-\BP_{\BM} \Bw^t\times \Bn , \Be_{\Bu}^t- \Be_{\widehat{\Bu}^t} \rangle_{\partial \Ct_h} + \langle \tau_t \big( (\boldsymbol{\Pi}_{\BU} \Bu)^t - \BP_{\BM} \Bu^t  \big), \Be_{\Bu}^t- \Be_{\widehat{\Bu}^t} \rangle_{\partial \Ct_h} \nn\\
&\quad+ \langle e_\sigma- e_{\widehat{\sigma}}, (  \boldsymbol{\Pi}_{\BU} \Bu- \BP_{\BM} \Bu^n )\cdot\Bn \rangle_{\partial\Ct_h} + \langle e_\sigma- e_{\widehat{\sigma}},\tau_n (  \Pi_Q \sigma - P_M \sigma) \rangle_{\partial\Ct_h} \nn\\
& = - \langle (\boldsymbol{\Pi}_{\BV} \Bw-\Bw)\times \Bn , \Be_{\Bu}^t- \Be_{\widehat{\Bu}^t} \rangle_{\partial \Ct_h}  + \langle \tau_t ( \boldsymbol{\Pi}_{\BU} \Bu- \Bu  ), \Be_{\Bu}^t- \Be_{\widehat{\Bu}^t} \rangle_{\partial \Ct_h} \nn\\
& \quad+  \langle e_\sigma- e_{\widehat{\sigma}}, (  \boldsymbol{\Pi}_{\BU} \Bu- \Bu )\cdot\Bn \rangle_{\partial\Ct_h} + \langle e_\sigma- e_{\widehat{\sigma}},\tau_n (  \Pi_Q \sigma -  \sigma) \rangle_{\partial\Ct_h},\label{error_eq5}
\end{align}
where the second equality is derived by the properties of $L^2$-projections $\BP_{\BM}$ and $P_M$. Based on (\ref{error_eq5}), taking the real part and imaginary part of the left-hand side of (\ref{error_eq5}) respectively, the estimates (\ref{pre_error_analysis_1}) and (\ref{pre_error_analysis_2}) can be obtained by the approximation properties of standard $L^2$-projections, the Young's inequality and the fact that $\sigma=0$. This completes the proof.
\end{proof}

Now we start to use the duality argument to get an estimate for $\Be_{\Bu}$. Given $\Be_{\Bu} \in \BL^2(\Omega)$, we introduce the first-order system of the dual problem (\ref{dual_problem}) with $\BJ = \Be_{\Bu}$:
\begin{subequations}
\label{dual_FOS_error}
\begin{align}
\Vi \boldsymbol{\Phi} - \curl \boldsymbol{\Psi} &= 0\qquad \rm{in}\ \Omega, \\
\curl \boldsymbol{\Phi} + \Vi\kappa^2 \boldsymbol{\Psi} + \nabla \varphi &=  \Be_{\Bu} \quad \ \rm{in}\ \Omega,\\
\label{dual_FOS_error_1}
\divv \boldsymbol{\Psi} &= 0\qquad \rm{in}\ \Omega,\\
\label{bc1_dual_error_FOS}
\boldsymbol{\Phi} \times \Bn -  \kappa \boldsymbol{\Psi}^t &= 0 \qquad \rm{on}\ \partial\Omega,\\
\varphi &= 0 \qquad \rm{on}\ \partial\Omega.
\end{align}
\end{subequations}
Similar to the estimates in (\ref{est_dual_3_p})-(\ref{est_dual_3_1}), we have 
\begin{align}
\label{est_dual_error_s}
&\|\varphi\|_{1,\Omega} \leq C \|\Be_{\Bu}\|_{\Ct_h},
\\
\label{est_dual_error_1_new}
&\| \boldsymbol{\Phi} \|_{\frac{1}{2}+\alpha,\Omega}  + \kappa \| \boldsymbol{\Psi} \|_{\frac{1}{2}+\alpha,\Omega} + \|\curl \boldsymbol{\Psi}\|_{\frac{1}{2}+\alpha,\Omega} + (1+\kappa)\kappa \| \boldsymbol{\Psi}\|_{0,\Omega}     \leq C (1+\kappa) \|\Be_{\Bu} \|_{\Ct_h}.
\end{align}

Next we first present an important equality.
\begin{lemma}
\label{lemma_equality_error_eu}
Let $(\boldsymbol{\Phi},\boldsymbol{\Psi},\varphi) $ be the solution of  the dual problem (\ref{dual_FOS_error}). It holds that
\begin{align}
\label{equality_error_eu}
\|  \Be_{\Bu} \|^2_{\Ct_h} = \sum^{5}_{k=1} E_k,
\end{align}
where 
\begin{align*}
E_1 &= \langle (\Be^t_{\Bu} -\Be_{\widehat{\Bu}^t} )\times \Bn, \boldsymbol{\Phi} - \boldsymbol{\Pi_V \Phi} \rangle_{\partial \Ct_h},\\
E_2 & = \langle (\Be^n_{\Bu} -\Be_{\widehat{\Bu}^n} )\cdot \Bn, \varphi - \Pi_Q \varphi\rangle_{\partial \Ct_h},\\
E_3&=\langle (\Be_{\widehat{\Bw}^t} - \Be_{\Bw}^t )\times \Bn,\boldsymbol{\Psi} - \boldsymbol{\Pi_U \Psi} \rangle_{\partial \Ct_h},\\
E_4 &= -\langle \Be_{\widehat{\Bw}^t} \times \Bn+ \kappa \Be_{\widehat{\Bu}^t} , \boldsymbol{\Psi}^t \rangle_{\partial \Omega} , \\
E_5 & = (e_\sigma, \divv (\boldsymbol{\Psi} - \boldsymbol{\Pi_U \Psi}  )  )_{\Ct_h} - \langle e_{\widehat{\sigma}}, (\boldsymbol{\Psi} - \boldsymbol{\Pi_U \Psi} )\cdot \Bn \rangle_{\partial \Ct_h}.
\end{align*}
\end{lemma}

\begin{proof}
By the dual problem (\ref{dual_FOS_error}), we have
\begin{align}
\|  \Be_{\Bu} \|^2_{\Ct_h} & = (\Be_{\Bu} , \curl \boldsymbol{\Phi} + \Vi\kappa^2 \boldsymbol{\Psi} + \nabla \varphi)_{\Ct_h} + (\Be_{\Bw} , \Vi \boldsymbol{\Phi} - \curl \boldsymbol{\Psi})_{\Ct_h}\nn\\
& = (\Be_{\Bu} , \curl \boldsymbol{\Pi_V \Phi} + \Vi\kappa^2 \boldsymbol{\Pi_U \Psi} + \nabla \Pi_Q \varphi )_{\Ct_h}  + (\Be_{\Bw} ,\Vi \boldsymbol{\Pi_V \Phi} - \curl \boldsymbol{\Pi_U \Psi})_{\Ct_h}\nn\\
&\quad \ + (\Be_{\Bu} ,\curl (\boldsymbol{\Phi} -\boldsymbol{\Pi_V \Phi}   ) )_{\Ct_h} + (\Be_{\Bu} ,\Vi\kappa^2(  \boldsymbol{\Psi}  - \boldsymbol{\Pi_U\Psi} ))_{\Ct_h} + (\Be_{\Bu} , \nabla(\varphi - \Pi_Q \varphi))_{\Ct_h}\nn\\
&\quad \ + (\Be_{\Bw} ,\Vi(\boldsymbol{\Phi} -\boldsymbol{\Pi_V\Phi} ))_{\Ct_h} - (\Be_{\Bw} , \curl (\boldsymbol{\Psi}-\boldsymbol{\Pi_U\Psi}  ))_{\Ct_h}.\label{err_equ_pre_0}
\end{align}
By the definitions of $\boldsymbol \Pi_U$ and $\boldsymbol \Pi_V$, we have $(\Be_{\Bu},\Vi\kappa^2(  \boldsymbol{\Psi}  - \boldsymbol{\Pi_U\Psi} ))_{\Ct_h} =0$ and $(\Be_{\Bw},\Vi(\boldsymbol{\Phi} -\boldsymbol{\Pi_V\Phi} ))_{\Ct_h}=0$. Similar to the derivations of (\ref{sta_equ_pre_1})-(\ref{sta_equ_pre_3}), we have
\begin{align}
(\Be_{\Bu},\curl (\boldsymbol{\Phi} -\boldsymbol{\Pi_V \Phi}   ) )_{\Ct_h} &= \langle \Be_{\Bu}^t \times \Bn, \boldsymbol{\Phi} -\boldsymbol{\Pi_V \Phi} \rangle_{\partial \Ct_h},\label{err_equ_pre_1}
\\
 (\Be_{\Bu}, \nabla(\varphi - \Pi_Q \varphi))_{\Ct_h} &= \langle \Be_{\Bu}^n \cdot \Bn, \varphi-\Pi_Q \varphi\rangle_{\partial \Ct_h} ,\label{err_equ_pre_2}
\\
- (\Be_{\Bw}, \curl (\boldsymbol{\Psi}-\boldsymbol{\Pi_U\Psi}  ))_{\Ct_h} &= - \langle \Be^t_{\Bw} \times \Bn, \boldsymbol{\Psi}-\boldsymbol{\Pi_U\Psi}\rangle_{\partial \Ct_h}.\label{err_equ_pre_3}
\end{align}
Taking $\Br_h =  \boldsymbol{\Pi_V \Phi}$ in the equation (\ref{error_eq_a}), noting that $\Be_{\widehat{\Bu}^t}\times \Bn$ is continuous across each interior face and using the boundary condition (\ref{bc1_dual_error_FOS}), we obtain
\begin{align}
(\Be_{\Bu}, \curl \boldsymbol{\Pi_V \Phi} )_{\Ct_h} &= (\Vi \Be_{\Bw},\boldsymbol{\Pi_V \Phi} )_{\Ct_h} + \langle \Be_{\widehat{\Bu}^t}\times \Bn, \boldsymbol{\Pi_V \Phi}\rangle_{\partial \Ct_h} \nn\\
& =  (\Vi \Be_{\Bw},\boldsymbol{\Pi_V \Phi} )_{\Ct_h} + \langle \Be_{\widehat{\Bu}^t}\times \Bn, \boldsymbol{\Pi_V \Phi} -  \boldsymbol{ \Phi} \rangle_{\partial \Ct_h} - \langle \Be_{\widehat{\Bu}^t}, \boldsymbol{ \Phi}\times\Bn\rangle_{\partial \Omega}\nn\\
& =  (\Vi \Be_{\Bw},\boldsymbol{\Pi_V \Phi} )_{\Ct_h} + \langle\Be_{\widehat{\Bu}^t}\times \Bn, \boldsymbol{\Pi_V \Phi} -  \boldsymbol{ \Phi} \rangle_{\partial \Ct_h} - \langle \Be_{\widehat{\Bu}^t}, \kappa \boldsymbol{\Psi}^t \rangle_{\partial \Omega}.\label{err_equ_pre_4}
\end{align}
Note that $\Be_{\widehat{\Bu}^n}\cdot \Bn$ is continuous across each interior face and $\varphi \in H^1_0(\Omega)$. We let $q_h = \Pi_Q \varphi$ in (\ref{error_eq_c}) to obtain
\begin{align}
(\Be_{\Bu},\nabla \Pi_Q \varphi)_{\Ct_h} = \langle\Be_{\widehat{\Bu}^n} \cdot \Bn, \Pi_Q \varphi \rangle_{\partial \Ct_h}  = \langle \Be_{\widehat{\Bu}^n}\cdot \Bn, \Pi_Q \varphi - \varphi \rangle_{\partial \Ct_h}  .\label{err_equ_pre_5}
\end{align}
We further take $\Bv_h = \boldsymbol{\Pi_U \Psi} $ in (\ref{error_eq_b}) to get
\begin{align}
&-( \Be_{\Bw},\curl  \boldsymbol{\Pi_U\Psi})_{\Ct_h} \nn\\
&= - \langle \Be_{\widehat{\Bw}^t}\times \Bn,\boldsymbol{\Pi_U\Psi} \rangle_{\partial \Ct_h} + ( \Vi \kappa^2 \Be_{\Bu},\boldsymbol{\Pi_U\Psi} )_{\Ct_h} - ( e_\sigma,\divv \boldsymbol{\Pi_U\Psi})_{\Ct_h} + \langle e_{\widehat{\sigma}}, \boldsymbol{\Pi_U\Psi} \cdot \Bn \rangle_{\partial \Ct_h} ,\nn\\
&= - \langle \Be_{\widehat{\Bw}^t}\times \Bn,\boldsymbol{\Pi_U\Psi} - \boldsymbol{\Psi}\rangle_{\partial \Ct_h}  - \langle \Be_{\widehat{\Bw}^t}\times \Bn,\boldsymbol{\Psi}^t\rangle_{\partial \Omega}
+ ( \Vi \kappa^2 \Be_{\Bu},\boldsymbol{\Pi_U\Psi} )_{\Ct_h} \nn\\
&\quad - ( e_\sigma,\divv \boldsymbol{\Pi_U\Psi} - \divv \boldsymbol{\Psi} )_{\Ct_h} + \langle e_{\widehat{\sigma}}, \boldsymbol{\Pi_U\Psi} \cdot \Bn -\boldsymbol{\Psi} \cdot \Bn  \rangle_{\partial \Ct_h} ,\label{err_equ_pre_6}
\end{align}
where the above second equality holds due to the fact that $\divv \boldsymbol{\Psi}=0$, $ \Be_{\widehat{\Bw}^t}\times \Bn$ and $e_{\widehat{\sigma}}$ are continuous across each interior face, and $e_{\widehat{\sigma}}=0$ on $\Ce^{\partial}_h$. Then, inserting (\ref{err_equ_pre_1})-(\ref{err_equ_pre_6}) into (\ref{err_equ_pre_0}) yields the result.
\end{proof}

Based on the above lemma, we can obtain the estimate for $\|\Be_{\Bu}\|_{\Ct_h}$.

\begin{lemma}\label{eu_pre}
If the regularity property (\ref{est_dual_error_s}) holds, and $\tau_t,\tau_n$ are chosen as in Lemma \ref{error_analysis_1},   we have
\begin{align}
\label{est_eu_result}
\|  \Be_{\Bu}  \|_{\Ct_h} \leq  C \big( R_{\Bw}  \| \Bw \|_{\frac{1}{2}+\alpha,\Omega}+  R_{\Bu} \| \Bu \|_{\frac{1}{2}+\alpha,\Omega}\big),
\end{align}
where $R_{\Bw}$ and $R_{\Bu}$ are defined as in Theorem \ref{error_thm}.
\end{lemma}

\begin{proof}
We need to derive the upper bounds for $E_1,\cdots,E_5$ in Lemma \ref{lemma_equality_error_eu}. By the Cauchy-Schwarz inequality and the approximation property of  $\boldsymbol{\Pi_V}$, we obtain
\begin{align*}
E_1 \leq C \| \tau_t^{\frac{1}{2}}(\Be^t_{\Bu} -\Be_{\widehat{\Bu}^t} ) \|_{\partial \Ct_h} \tau_t^{-\frac{1}{2}} (\frac{h}{p})^{\alpha} \| \boldsymbol{\Phi} \|_{\frac{1}{2}+\alpha,\Omega} \leq C\| \tau_t^{\frac{1}{2}}(\Be^t_{\Bu} -\Be_{\widehat{\Bu}^t})  \|_{\partial \Ct_h} (1+\kappa)(\frac{h}{p} )^{\frac{1}{2}+\alpha}\| \Be_{\Bu} \|_{\Ct_h}.
\end{align*}
By the identity (\ref{error_eq4}) for $ (\Be^n_{\Bu} -\Be_{\widehat{\Bu}^n} )\cdot \Bn$ and the fact that $\sigma=0$, we can derive that
\begin{align*}
E_2& = ((  \boldsymbol{\Pi}_{\BU} \Bu- \Bu+\Bu^n-\BP_{\BM} \Bu^n )\cdot\Bn- \tau_n( e_\sigma- e_{\widehat{\sigma}})+\tau_n (  \Pi_Q \sigma  - P_M \sigma) , \varphi - \Pi_Q \varphi\rangle_{\partial \Ct_h},\\
&\leq C\big( (\frac{h}{p})^{\alpha} \| \Bu \|_{\frac{1}{2}+\alpha,\Omega} 
+\tau_n^{\frac{1}{2}} \| \tau_n^{\frac{1}{2}} (e_\sigma -e_{\widehat{\sigma}})\|_{\partial \Ct_h} \big) (\frac{h}{p})^{\frac{1}{2}} \| \varphi \|_{1,\Omega}\\
&\leq C\big((\frac{h}{p})^{\frac{1}{2}+\alpha} \| \Bu \|_{\frac{1}{2}+\alpha,\Omega} + \frac{ (1+\kappa)^{\frac{1}{2}}h}{p} \| \tau_n^{\frac{1}{2}} (e_\sigma -e_{\widehat{\sigma}})\|_{\partial \Ct_h} \big) \| \Be_{\Bu} \|_{\Ct_h}
\end{align*}
Moreover, by the identity (\ref{error_eq3}) for $(\Be^t_{\Bw}-\Be_{\widehat{\Bw}^t})\times \Bn $ and the triangular inequality, we get
\begin{align*}
E_3&=-\langle \tau_t ( \Be^t_{\Bu} -\Be_{\widehat{\Bu}^t} ),\boldsymbol{\Psi} - \boldsymbol{\Pi_U \Psi} \rangle_{\partial \Ct_h}+\langle  \BP_{\BM}\Bw^t  \times \Bn-(\boldsymbol{\Pi_V \Bw} )^t \times \Bn,\boldsymbol{\Psi} - \boldsymbol{\Pi_U \Psi} \rangle_{\partial \Ct_h}\\
&\quad +\tau_t \langle  ( \boldsymbol{\Pi_U \Bu} )^t - \BP_{\BM}\Bu^t ,\boldsymbol{\Psi} - \boldsymbol{\Pi_U \Psi} \rangle_{\partial \Ct_h} \\
&\leq C \Big( \frac{(1+\kappa)}{\kappa}(\frac{h}{p})^{\alpha-\frac{1}{2}} \| \tau_t^{\frac{1}{2}}(\Be^t_{\Bu} -\Be_{\widehat{\Bu}^t}  )\|_{\partial \Ct_h} + \frac{(1+\kappa)}{\kappa}(\frac{h}{p})^{2\alpha} \| \Bw \|_{\frac{1}{2}+\alpha,\Omega}\\
&\quad +\frac{(1+\kappa)}{\kappa}(\frac{h}{p})^{2\alpha-1} \|\Bu\|_{\frac{1}{2}+\alpha,\Omega}  \Big) \| \Be_{\Bu} \|_{\Ct_h}.
\end{align*}
By the boundary condition (\ref{error_eq_e}), we have $E_4 =-\langle \Be_{\widehat{\Bw}^t} \times \Bn+ \kappa \Be_{\widehat{\Bu}^t} , \boldsymbol{P_M}\boldsymbol{\Psi}^t \rangle_{\partial \Omega}=0 $. Applying integration by parts,  we obtain the estimate for $E_5$ as follows:
\begin{align*}
E_5 & = ( \nabla e_\sigma, \boldsymbol{ \Psi-\Pi_U \Psi }  )+\langle e_\sigma-e_{\widehat{\sigma}}, (  \boldsymbol{ \Psi-\Pi_U \Psi }  )\cdot \Bn\rangle_{\partial \Ct_h}  \\
&= \langle e_\sigma-e_{\widehat{\sigma}}, (  \boldsymbol{ \Psi-\Pi_U \Psi }  )\cdot \Bn\rangle_{\partial \Ct_h}  \leq C  \frac{(1+\kappa)^{\frac{1}{2}}}{\kappa}(\frac{h}{p})^{\alpha-\frac{1}{2}}\| \tau_n^{\frac{1}{2}} (e_\sigma -e_{\widehat{\sigma}})\|_{\partial \Ct_h} \| \Be_{\Bu} \|_{\Ct_h}.
\end{align*}
Finally, combining the above estimates for $E_1,\cdots,E_5$ and the estimate (\ref{pre_error_analysis_1}), we can conclude the result.
\end{proof}

We can now give the proof of Theorem \ref{error_thm}.
\begin{proof} (Proof of Theorem \ref{error_thm}) By the triangular inequality, we have
\[
\|\Bu-\Bu_h\|_{\Ct_h} \leq \| \Bu - \boldsymbol{\Pi_U u} \|_{\Ct_h} + \|\Be_{\Bu}\|_{\Ct_h}.
\]
The error estimate (\ref{error_ut_h1}) can be obtained by the approximation property of $\boldsymbol{\Pi_U}$ and the estimate (\ref{est_eu_result}) for $\|\Be_{\Bu}\|_{\Ct_h}$. Similarly, (\ref{error_wt_h1}) can be obtained by the triangular inequality, the approximation property of $\boldsymbol{\Pi_V}$, (\ref{pre_error_analysis_2}) and (\ref{est_eu_result}). 
\end{proof}

\begin{remark}
Besides we get the error estimate for $\|\tau_n^{\frac{1}{2}}  ( e_\sigma- e_{\widehat{\sigma}} ) \|_{\partial \Ct_h}$ in (\ref{pre_error_analysis_1}), we can also obtain the error estimate for $\|\nabla e_\sigma\|_{\Ct_h}$. Actually, this can be similarly derived as the stability estimate for $\|\nabla \sigma_h\|_{\Ct_h}$ (cf. Remark \ref{remark_sta_1}) by taking $\Bv_h = \nabla e_\sigma$ in the error equation (\ref{error_eq_b}). Then the error estimate for $\| \nabla(\sigma-\sigma_h) \|_{\Ct_h}$ can be further deduced by the triangular inequality. When $\kappa h/p \leq C_0$, one may tune the parameters $\tau_t$ and $\tau_n$ (cf. Remark \ref{remark_sta_1}) and get the error estimates.
\end{remark}

\section{Stability and error estimates for ideal case}\label{ideal_case}
In this section, we consider the stability estimates and error estimates of the HDG method (\ref{discrete_mixed_form}) under some ideal assumptions of the problem (\ref{pde_original}) and the dual problem (\ref{dual_problem}). We assume that when $\Omega$ is a smooth star-shaped domain, the solutions of the first-order system (\ref{pde_mixed_first_order})  satisfy that $\Bu \in \BH^2(\Omega)$ and $\Bw \in \BH^1(\Omega)$. When $\Bf $ is divergence-free and $\Bg\in
\BH_T^{\frac{1}{2}}(\partial\Omega):=\{\Bg\in
[H^{\frac{1}{2}}(\partial\Omega)]^3,\ \Bg\cdot\boldsymbol{\Bn}=0\
{\textrm {on}}\ \partial \Omega\}$, we assume the following estimate holds, which has been mentioned in \cite{FW2014} that 
\[
\| \Bu \|_{2,\Omega} + \|\Bw\|_{1,\Omega} \leq C ( 1 + \kappa ) \BM(\Bf,\Bg) + C\|\Bg\|_{\frac{1}{2},\partial \Omega},
\] 
where $\BM(\Bf,\Bg) = \|\Bf\|_{0,\Omega} +\|\Bg\|_{0,\partial \Omega} $. In this ideal case, we can also assume the solution of the dual problem (\ref{dual_problem}) satisfies that $\boldsymbol{\Psi} \in \BH^2(\Omega)$, the estimate (\ref{est_dual_1}) holds with $\alpha=\frac{1}{2}$ and there also holds
\begin{align}
\label{est_dual_2}
\|\boldsymbol{\Psi}\|_{2,\Omega} \leq C (1+\kappa) \| \BJ-\nabla \varphi \|_{0,\Omega} \leq C( 1+\kappa) \| \BJ\|_{0,\Omega}.
\end{align}
We assume the approximation results of $L^2$-projections in (\ref{es_pj_1})-(\ref{es_pj_1}) still hold, then we have the following stability estimates and error estimates for the HDG method (\ref{discrete_mixed_form}).
\begin{lemma}
\label{stability_thm0} 
We assume that (\ref{est_dual_1}) holds with $\alpha = \frac{1}{2} $ and (\ref{est_dual_2}) also holds true. Let $(\Bw_h,\Bu_h,\widehat{\Bu}^t_h,\sigma_h, \widehat{\sigma}_h)$ be the solution of the problem (\ref{discrete_mixed_form}). We have
\begin{align}
\label{stability_u_h0}
&\| \Bu_h \|_{\Ct_h} \leq C\Big(  \widetilde{C}^2_{\rm stab}\| 
\frac{\Bf}{\kappa} \|_{0,\Omega}   + \widetilde{C}_{\rm stab}
\|\frac{\Bg}{\kappa}\|_{0,\partial \Omega}   \Big) ,\\
\label{stability_w_h0}
&\| \Bw_h \|_{\Ct_h} \leq C\Big(  (\frac{1}{\kappa}+\widetilde{C}^2_{\rm stab}) \| \Bf \|_{0,\Omega}   +(\frac{1}{\kappa^{\frac{1}{2}}}+\widetilde{C}_{\rm stab})  \|\Bg\|_{0,\partial \Omega}   \Big), \\
\label{stability_ut_h0}
&\| \widehat{ \Bu}^t_h \|_{\partial \Ct_h} \leq C\big( (\frac{\kappa h}{p})^{\frac{1}{2}}
+ ph^{-\frac{1}{2}} \big)\Big( \widetilde{C}^2_{\rm stab}\| 
\frac{\Bf}{\kappa} \|_{0,\Omega}   + \widetilde{C}_{\rm stab}
\|\frac{\Bg}{\kappa}\|_{0,\partial \Omega}   \Big),
\end{align}
where $\widetilde{C}_{\rm stab} := 1 + \frac{(1+\kappa)\kappa^{\frac{1}{2}} h}{p}+\frac{(1+\kappa)^{\frac{1}{2}}\kappa^{\frac{1}{2}} h}{p}$.
\end{lemma}
\begin{proof}
In order to get the upper bound for $\| \Bu_h \|_{\Ct_h}$, indeed, it only needs to bound the terms $T_1,\cdots,T_6$ as in the proof of Theorem \ref{stability_thm}. When (\ref{est_dual_1}) holds with $\alpha = \frac{1}{2} $ and (\ref{est_dual_2}) also holds, we have the following regularity estimate for the dual problem (\ref{dual_FOS}),
\begin{align}
\| \boldsymbol{\Phi}\|_{1,\Omega}  + \| \boldsymbol{\Psi} \|_{2,\Omega} + \kappa \| \boldsymbol{\Psi} \|_{1,\Omega} + \|\curl \boldsymbol{\Psi}\|_{1,\Omega} + \kappa(1+\kappa) \| \boldsymbol{\Psi}\|_{0,\Omega} \leq C(1+\kappa) \|\Bu_h\|_{\Ct_h}.\nn
\end{align}
By the above regularity estimate, we have
\begin{align*}
T_1& \leq  C \| \tau_t^{\frac{1}{2}} (\Bu^t_h- \widehat{\Bu}^t_h)  \|_{\partial \Ct_h} \tau_t^{-\frac{1}{2}} (\frac{h}{p})^{\frac{1}{2}} (1+\kappa) \|\Bu_h\|_{\Ct_h},\\
T_3 &\leq  C \| \tau_t^{\frac{1}{2}} (\Bu^t_h- \widehat{\Bu}^t_h)  \|_{\partial \Ct_h} \tau_t^{\frac{1}{2}} (\frac{h}{p})^{\frac{3}{2}} (1+\kappa) \|\Bu_h\|_{\Ct_h},\\
T_6 & \leq C \| \tau^{\frac{1}{2}}_n (\sigma_h - \widehat{\sigma}_h) \|_{\partial \Ct_h} \tau^{-\frac{1}{2}}_n  (\frac{h}{p})^{\frac{3}{2}} (1+\kappa)  \|\Bu_h\|_{\Ct_h}
\end{align*}
and the estimates for $T_2,T_4,T_5$ are the same as the estimates in the proof of Theorem \ref{stability_thm}. Combining the estimates for $T_1,\cdots, T_6$ again, we obtain
\begin{align}
\|\Bu_h\|^2_{\Ct_h} & \leq C \kappa^{-1}\left(\| \Bf  \|_{0,\Omega}+  \|\Bg\|_{0,\partial \Omega} \right)\|\Bu_h\|_{\Ct_h} \nn\\
&\quad + C
\big(\tau_t^{-\frac{1}{2}} (\frac{h}{p})^{\frac{1}{2}} +  \tau_t^{\frac{1}{2}} (\frac{h}{p})^{\frac{3}{2}}  \big)(1+\kappa) \| \tau_t^{\frac{1}{2}} (\Bu^t_h- \widehat{\Bu}^t_h)  \|_{\partial \Ct_h} \|\Bu_h\|_{\Ct_h} \nn\\
& \quad + C \big( \tau_n^{\frac{1}{2}} (\frac{h}{p})^{\frac{1}{2}}+\tau^{-\frac{1}{2}}_n  (\frac{h}{p})^{\frac{3}{2}} (1+\kappa) \big)  \| \tau^{\frac{1}{2}}_n (\sigma_h - \widehat{\sigma}_h) \|_{\partial \Ct_h}    \|\Bu_h\|_{\Ct_h}. \label{sta_uh_new0}
\end{align}
Here we choose $\tau_t = \frac{p}{h}$ and $\tau_n = \frac{(1+\kappa) h}{p}$. Then, the stability estimate (\ref{stability_u_h0}) can be obtained by (\ref{sta_uh_new0}), (\ref{estimate_mixed_1}) and the Young's inequality, and the stability estimates (\ref{stability_w_h0}) and (\ref{stability_ut_h0}) can be further derived as the analysis in Theorem \ref{stability_thm}.
\end{proof}

\begin{lemma}\label{error_ideal_lemma}
We assume that (\ref{est_dual_1}) holds with $\alpha = \frac{1}{2} $ and (\ref{est_dual_2}) also holds true. Let $(\Bw_h,\Bu_h,\widehat{\Bu}^t_h,\sigma_h, \widehat{\sigma}_h)$ be the solution of the problem (\ref{discrete_mixed_form}). We have
\begin{align}
\label{error_eu_thm0}
&\| \Bu - \Bu_h \|_{\Ct_h} \leq C \big( \widetilde{R}_{\Bw}  \| \Bw \|_{t,\Omega}+  \widetilde{R}_{\Bu} \| \Bu \|_{s,\Omega}\big) ,\\
\label{error_ew_thm0}
&\| \Bw - \Bw_h \|_{\Ct_h} \leq C \Big( \big(  \frac{h^t}{p^t}+\kappa \widetilde{R}_{\Bw}\big) \| \Bw \|_{t,\Omega}+\big( (1+(1+\kappa)^{-\frac{1}{2}}) \frac{h^{s-1}}{p^{s-1}}+ \kappa \widetilde{R}_{\Bu} \big) \| \Bu \|_{s,\Omega} \Big),
\end{align}
where $s\geq 1, t\geq 1$, $\widetilde{R}_{\Bw}:=\frac{(1+\kappa)^{\frac{1}{2}} h^{t+1} }{p^{t+1}}  +\frac{(1+\kappa) h^{t+1} }{p^{t+1}}$ and $\widetilde{R}_{\Bu}: = \frac{h^s}{p^s}+ \frac{(1+\kappa)^{\frac{1}{2}} h^{s} }{p^{s}}  +\frac{(1+\kappa) h^{s} }{p^{s}} $.
\end{lemma}

\begin{proof}
When (\ref{est_dual_1}) holds with $\alpha = \frac{1}{2} $ and (\ref{est_dual_2}) also holds, we have the following regularity estimate for the dual problem (\ref{dual_FOS_error}),
\begin{align*}
\| \boldsymbol{\Phi} \|_{1,\Omega} + \| \boldsymbol{\Psi} \|_{2,\Omega} + \kappa \| \boldsymbol{\Psi} \|_{1,\Omega} + \|\curl \boldsymbol{\Psi}\|_{1,\Omega} + (1+\kappa)\kappa \| \boldsymbol{\Psi}\|_{0,\Omega}     \leq C (1+\kappa) \|\Be_{\Bu} \|_{\Ct_h}.
\end{align*}
Repeating the similar estimates in Lemma \ref{eu_pre} for $E_1,\cdots,E_5$ and using (\ref{pre_error_analysis_1}), we obtain
\begin{align}
\|  \Be_{\Bu}  \|_{\Ct_h} \leq  C \big( \widetilde{R}_{\Bw}  \| \Bw \|_{t,\Omega}+  \widetilde{R}_{\Bu} \| \Bu \|_{s,\Omega}\big).\label{last_est0}
\end{align}
Then the error estimate (\ref{error_eu_thm0}) is obtained directly by the triangular inequality, the approximation property of $\boldsymbol{\Pi_U}$ and the above estimate. The error estimate  (\ref{error_ew_thm0}) can also be obtained by the triangular inequality, the approximation property of $\boldsymbol{\Pi_V}$, (\ref{pre_error_analysis_2}) and (\ref{last_est0}).
\end{proof}

\begin{remark}
By Lemma \ref{error_ideal_lemma}, under the assumptions made in the section, we have
\begin{align*}
\| \Bu - \Bu_h \|_{\Ct_h} &\leq C \big( \frac{\kappa h^2}{p^2} + \frac{\kappa^{\frac{3}{2}} h^2}{p^2} + \frac{\kappa^2 h^2}{p^2}   \big) \widehat{\BM}(\Bf,\Bg) \leq C(  \frac{\kappa h^2}{p^2}  + \frac{\kappa^2 h^2}{p^2} ),\\
\| \Bw - \Bw_h \|_{\Ct_h}& \leq C \big( \frac{\kappa h}{p} + \frac{\kappa^2 h^2}{p^2} + \frac{\kappa^{\frac{5}{2}} h^2}{p^2}  + \frac{\kappa^3 h^2}{p^2}   \big) \widehat{\BM}(\Bf,\Bg) \leq C( \frac{\kappa h}{p} +  \frac{\kappa^3 h^2}{p^2} ).
\end{align*}
Here $\kappa > 1$ and $\widehat{\BM}(\Bf,\Bg)  = {\BM}(\Bf,\Bg) + \frac{1}{\kappa} \|\Bg\|_{\frac{1}{2},\partial \Omega} $. The above estimates indicate that the error $\| \Bw - \Bw_h \|_{\Ct_h} $ can not be controlled by $\frac{\kappa h}{p}$, and the pollution term is of order $O( \frac{\kappa^3 h^2}{p^2})$. This provides evidence of the existence of the so-called ``pollution effect''. When $\frac{\kappa^3 h^2}{p^2}\leq C$, the discrete stability estimates for $\|\Bu_h\|_{\Ct_h}$ and $\|\Bw_h\|_{\Ct_h}$ can be improved as $ \|\Bu_h\|_{\Ct_h} \leq \frac{C}{\kappa} {\BM}(\Bf,\Bg)$ and  $\|\Bw_h\|_{\Ct_h} \leq C {\BM}(\Bf,\Bg)$.
\end{remark}

\section{Numerical results}
In this section, we present numerical results of the HDG method for the following time-harmonic Maxwell problem (cf. \cite{FW2014}) in a unit cube $\Omega=[0,1]\times[0,1]\times[0,1] $:
\begin{align*}
{\bf curl}\, {\bf curl}\, {\Bu}-\kappa^2 {\Bu}&={\bf 0 } \qquad {\rm in }\ \Omega,\\
{\bf curl}\, {\Bu}\times \boldsymbol{n}
-{\bf i}\kappa {\Bu}^t&=\widetilde{\Bg} \qquad {\rm on }\
\partial\Omega.
\end{align*}
Here  $\widetilde{\Bg}$ is chosen such that the exact solution is given by
$$\Bu=(e^{{\bf i}\kappa z},e^{{\bf i}\kappa x},e^{{\bf i}\kappa y})^T.$$

The time-harmonic Maxwell problem (\ref{pde_original}) is an approximation of electromagnetic scattering problem with time dependence $e^{{\bf i} \omega t}$, where $\omega$ is frequency. If the problem is proposed with time dependence $e^{-{\bf i} \omega t}$, then the sign before ${\bf i}$ in (\ref{BC-PDE}) is negative. The analysis of the HDG method in this paper fits well for both of cases. 
In the following experiment, we apply the HDG method with
piecewise linear (HDG-P1), piecewise quadratic (HDG-P2) and
piecewise cubic (HDG-P3) finite element spaces respectively to the second case. For the fixed wave
number $\kappa$, we first show the dependence of the convergence of
$\|\Bu-\Bu_h\|_{0,\Omega}$ and $\|\Bw-\Bw_h\|_{0,\Omega}$ on
polynomial order $p$ and mesh size $h$. Figure \ref{fig1} displays
the above  errors for $\kappa=20$ by the HDG-P1, HDG-P2, and HDG-P3
approximations. The pollution errors always appear on
the coarse meshes. However, we find that the errors converge almost
in $O(\kappa h^2/p^2)$ on the fine meshes, which is a little better than the
theoretical prediction for $\|\Bw-\Bw_h\|_{0,\Omega}$. On the other
hand, for the cases of $\kappa=30$ and $\kappa=50$, Figure \ref{fig2}
shows that the errors of $\|\Bu-\Bu_h\|_{0,\Omega}$ and
$\|\Bw-\Bw_h\|_{0,\Omega}$ always decrease for high order polynomial
approximations.

\begin{figure}[htbp]
\centering
    \includegraphics[width=3in,height=1.6in]{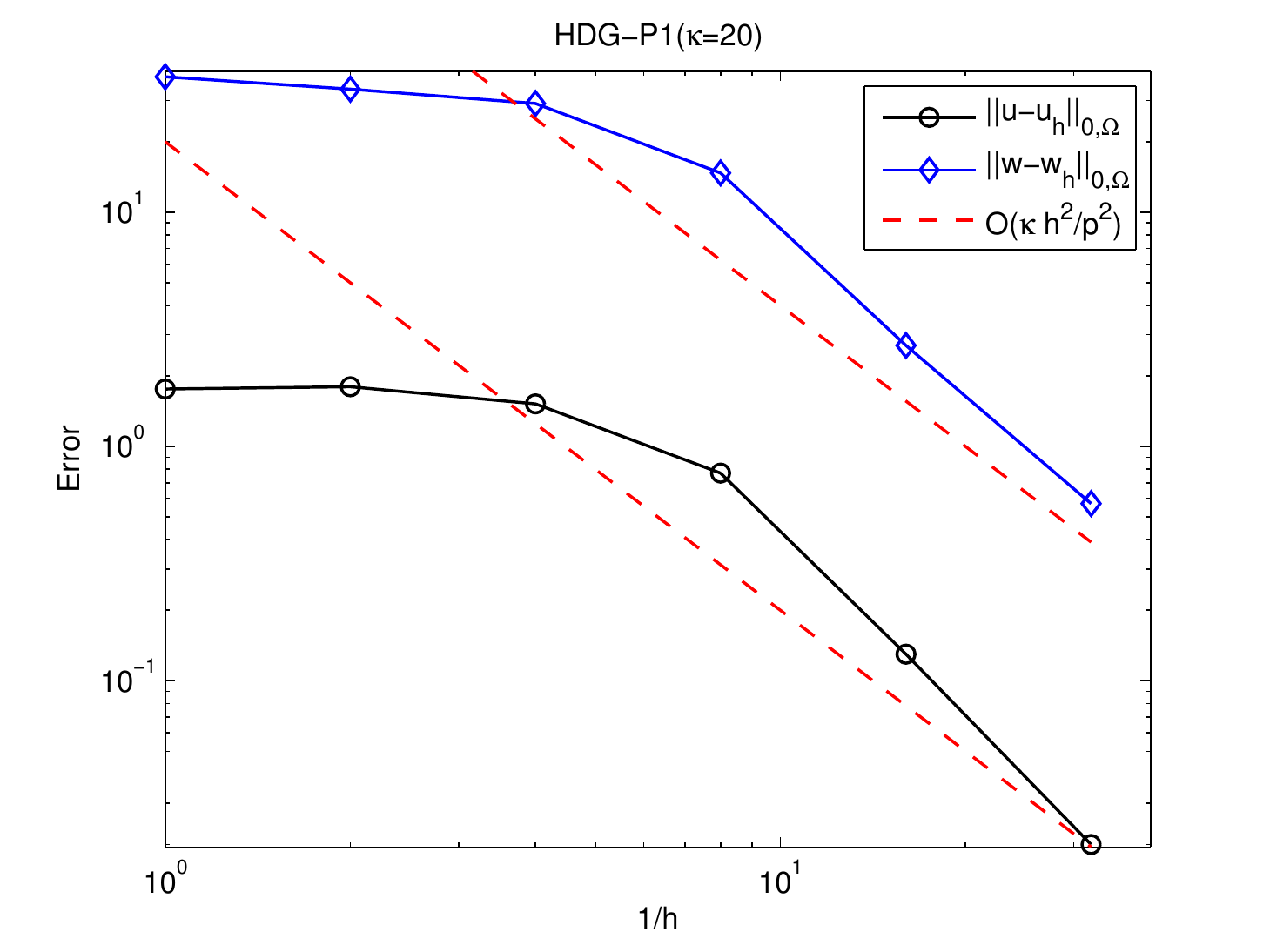}
    \includegraphics[width=3in,height=1.6in]{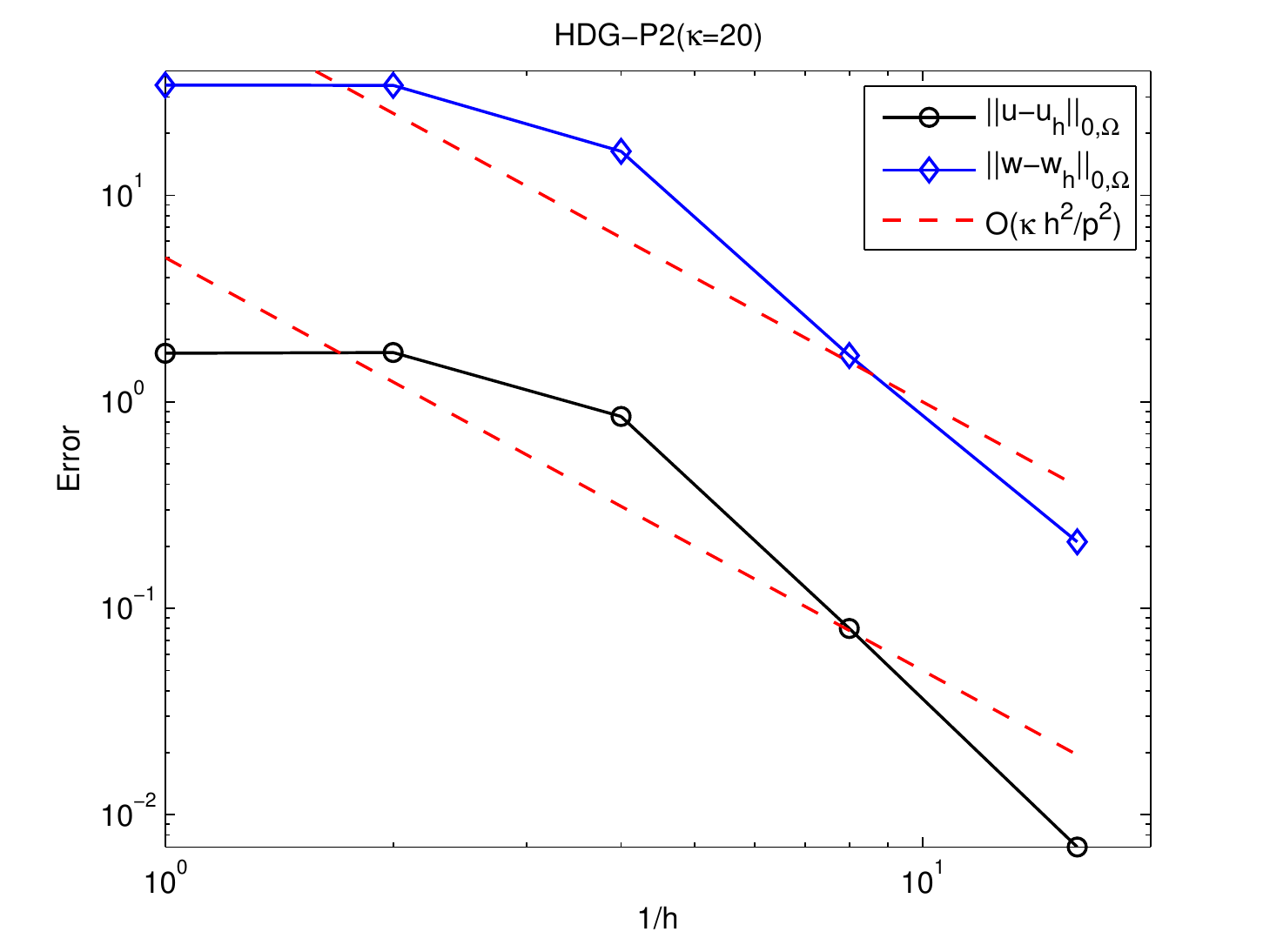}
    \includegraphics[width=3in,height=1.6in]{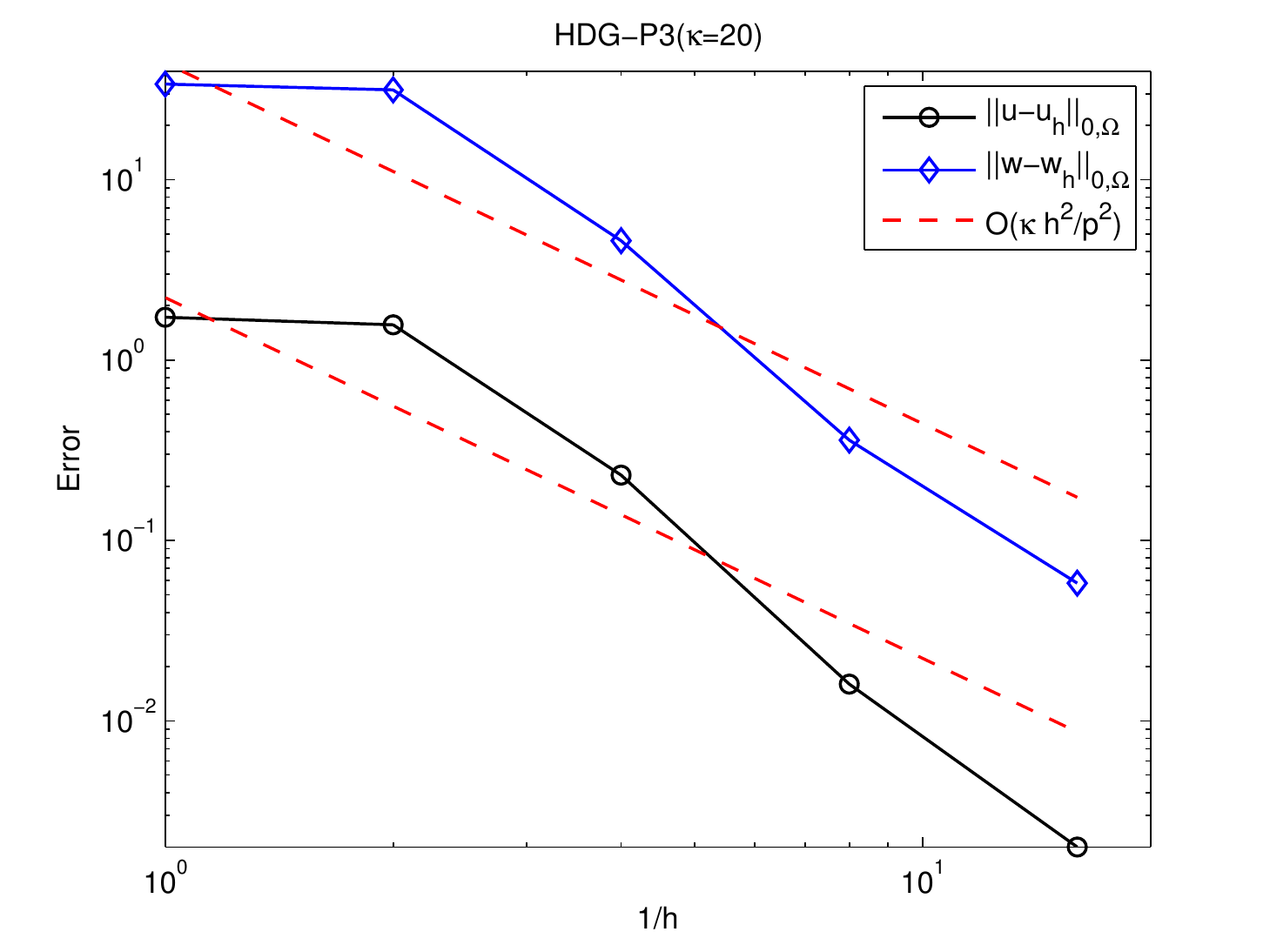}
    \caption{\footnotesize Errors of $\|\Bu-\Bu_h\|_{0,\Omega}$ and $\|\Bw-\Bw_h\|_{0,\Omega}$ for $\kappa=20$ by the HDG-P1, HDG-P2 and HDG-P3 approximations.}\label{fig1}
\end{figure}

\begin{figure}[htbp]
\centering
    \includegraphics[width=3in,height=1.6in]{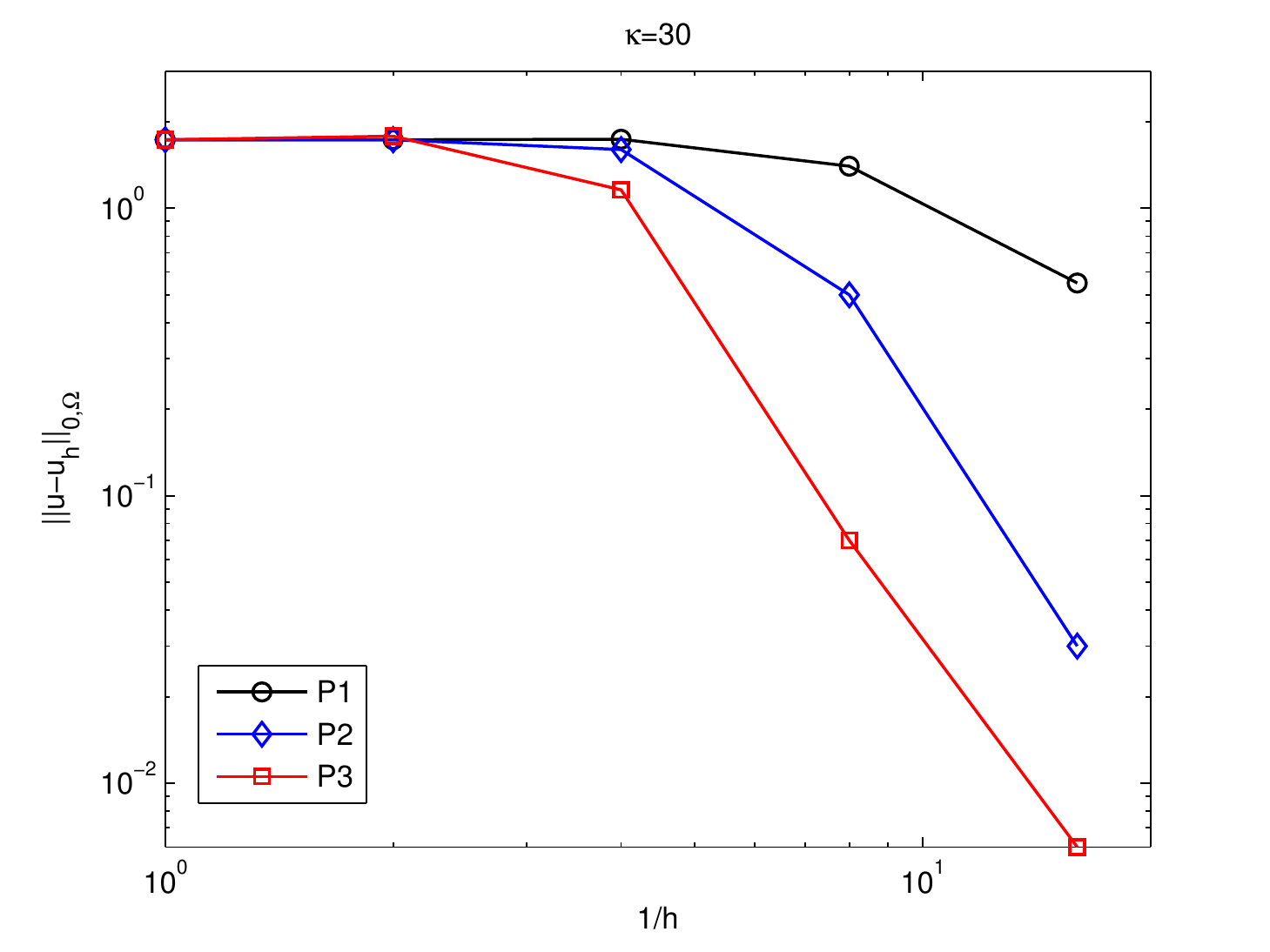}
    \includegraphics[width=3in,height=1.6in]{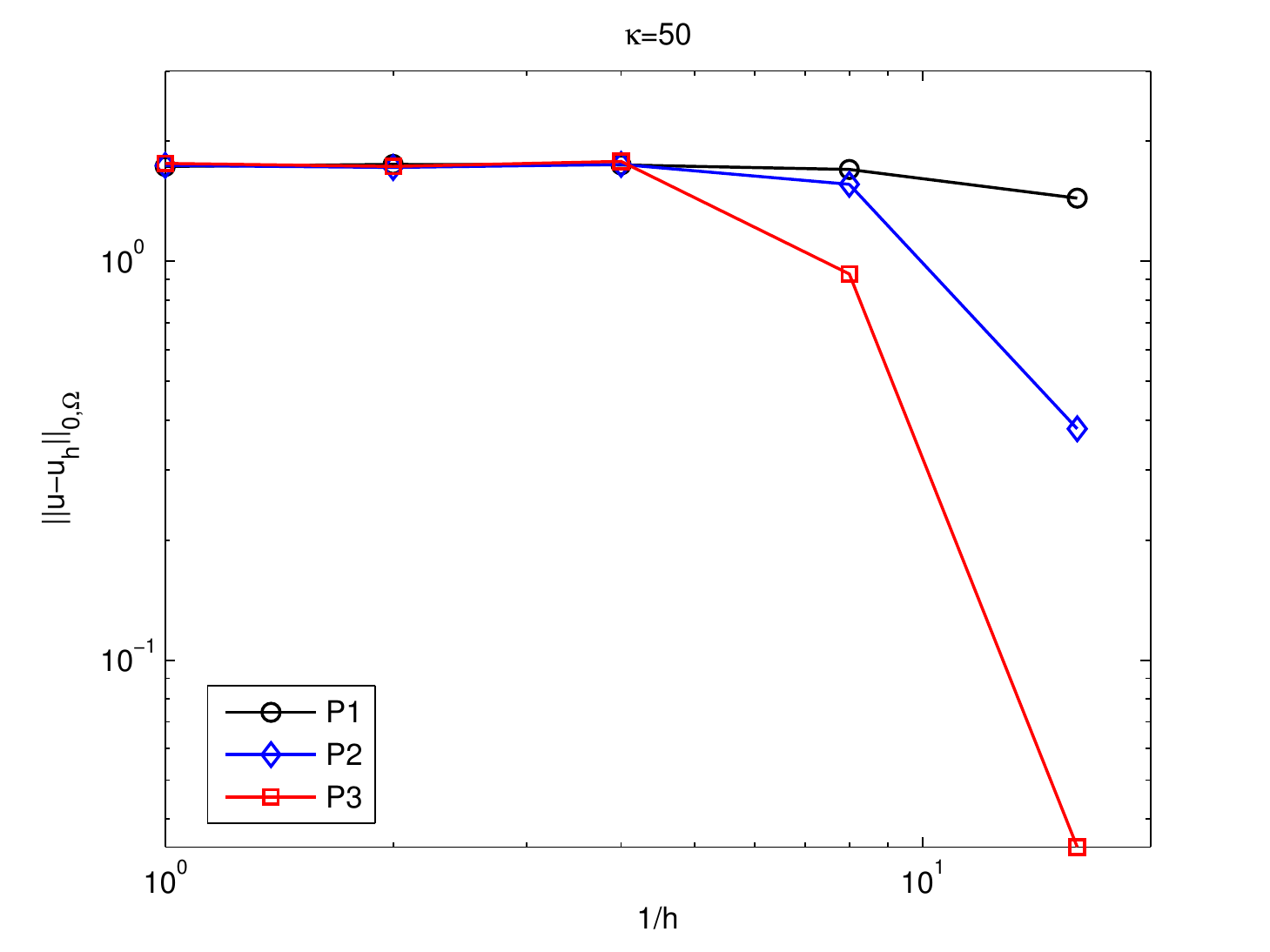}
    \includegraphics[width=3in,height=1.6in]{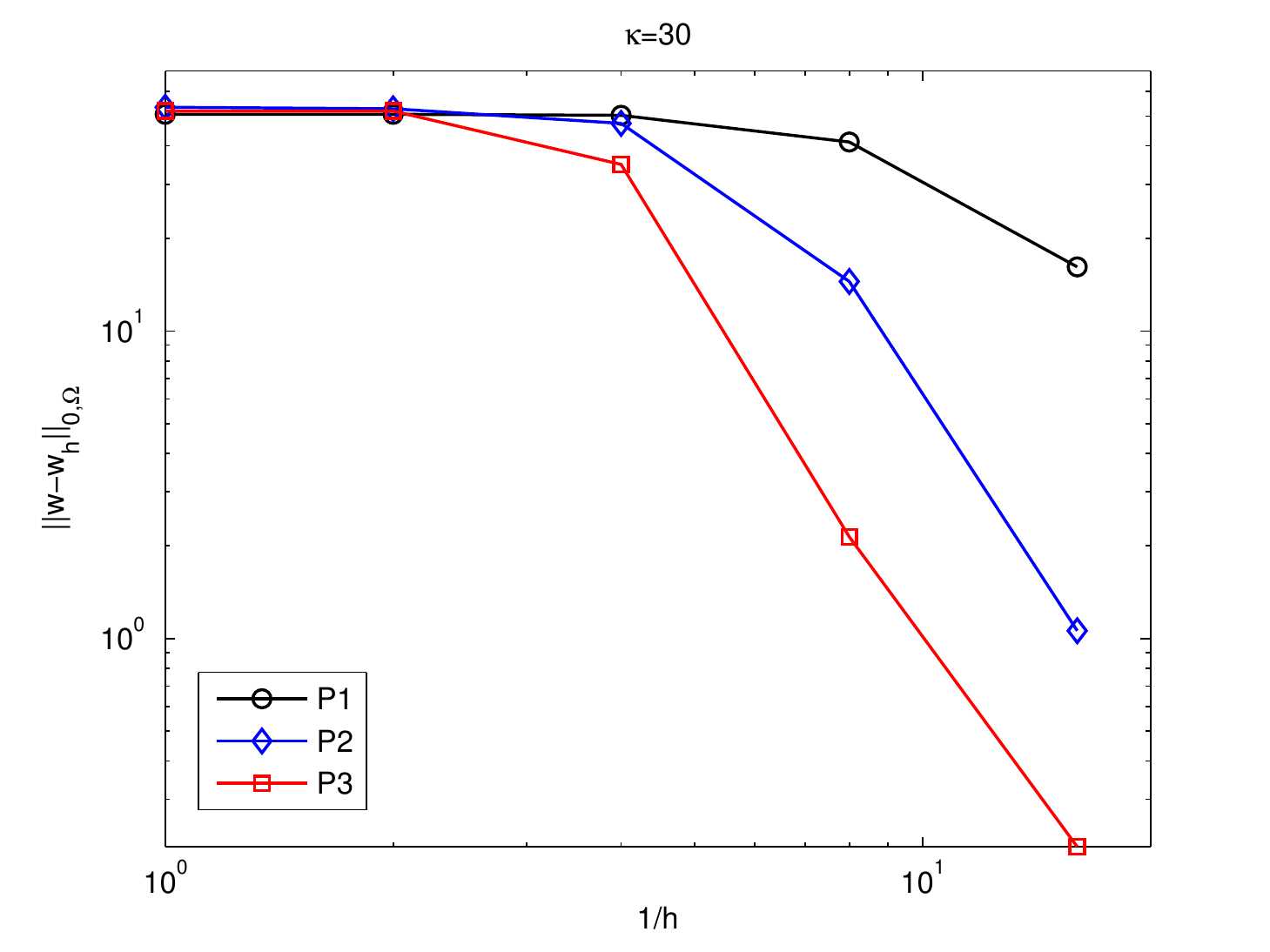}
     \includegraphics[width=3in,height=1.6in]{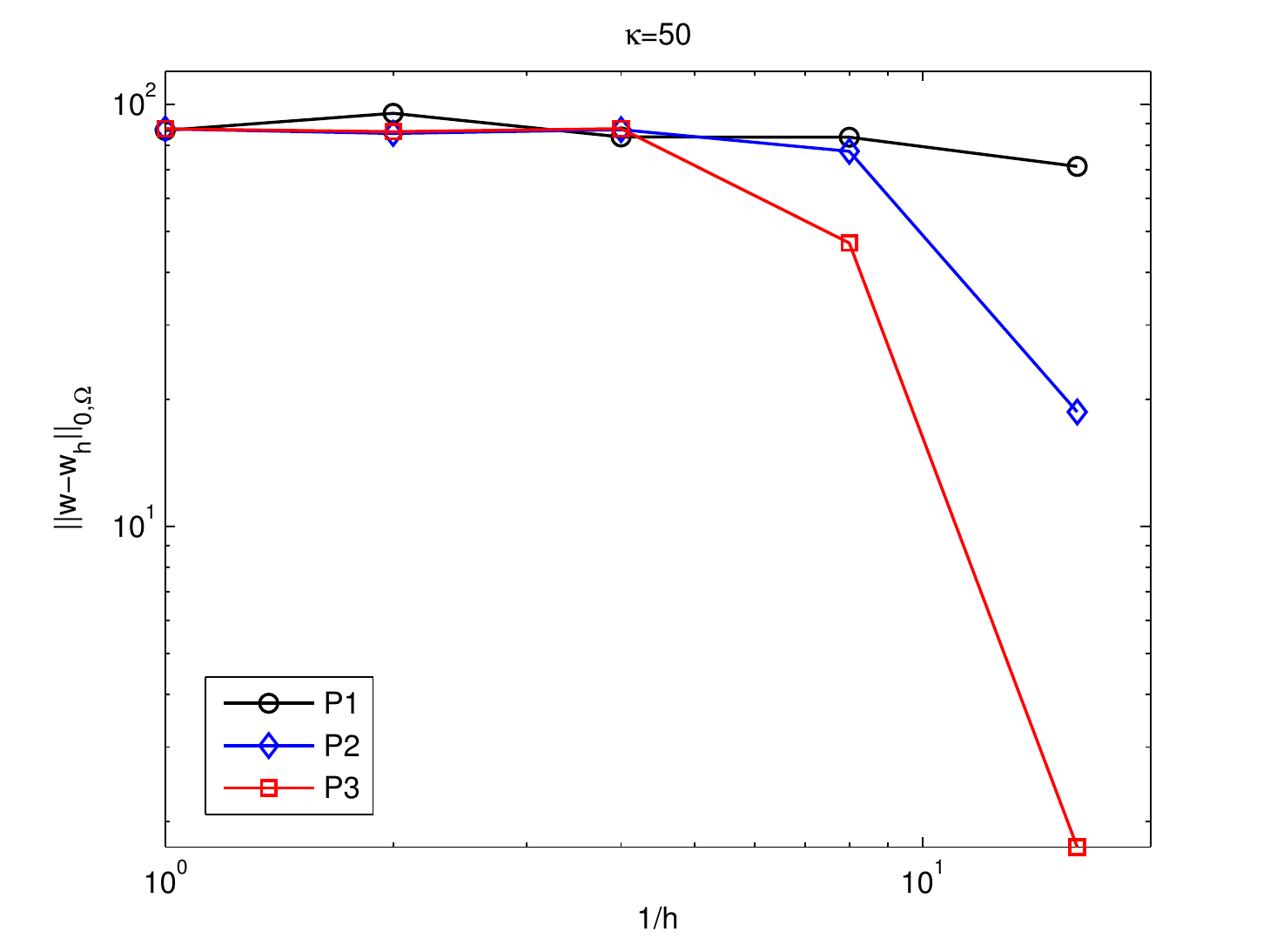}
    \caption{\footnotesize Errors of $\|\Bu-\Bu_h\|_{0,\Omega}$ and $\|\Bw-\Bw_h\|_{0,\Omega}$ for $\kappa=30$ and $\kappa=50$ by the HDG-P1,HDG-P2 and HDG-P3 approximations.}\label{fig2}
\end{figure}

Figure \ref{fig3} displays the relative errors $\|\Bu-\Bu_h\|_{0,\Omega}/\|\Bu\|_{0,\Omega}$ and $\|\Bw-\Bw_h\|_{0,\Omega}/\|\Bw\|_{0,\Omega}$ for the HDG-P1 approximation according to different mesh size conditions. The left graph of Figure \ref{fig3} shows the relationship between the relative errors and the wave number $\kappa$ under the mesh condition $\kappa h = 2$ for the HDG-P1 approximation. We observe that the relative errors
cannot be controlled by $\kappa h $ and increase with $\kappa $,
which indicates the existence of the pollution error. The right
graph of Figure \ref{fig3} shows the relative errors of the HDG-P1
approximation under the mesh condition $\kappa^3h^2 = 2$. It shows that
under this mesh condition, the relative errors do not increase with
$\kappa$.

\begin{figure}[htbp]
\centering
    \includegraphics[width=3in,height=1.6in]{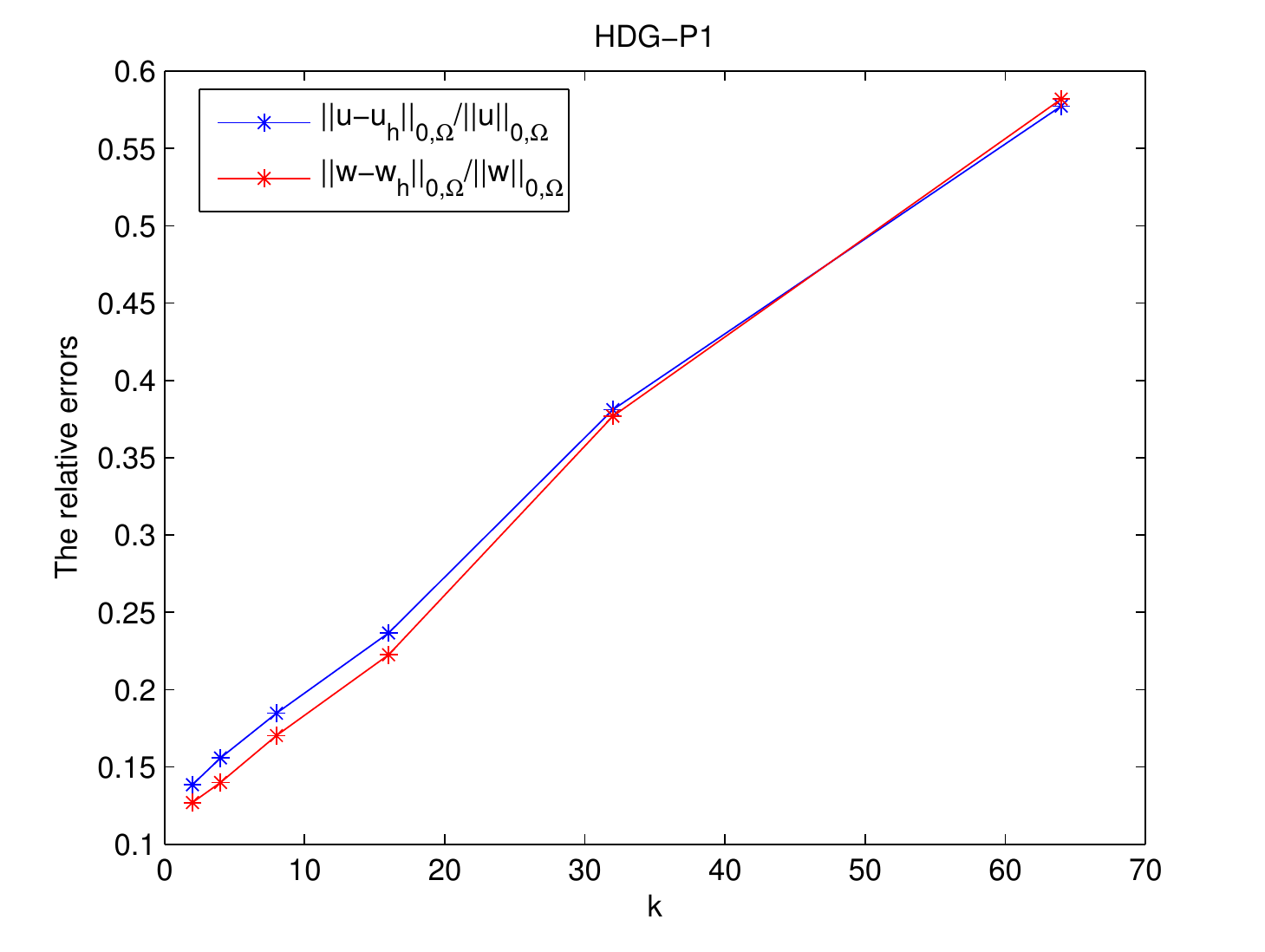}
    \includegraphics[width=3in,height=1.6in]{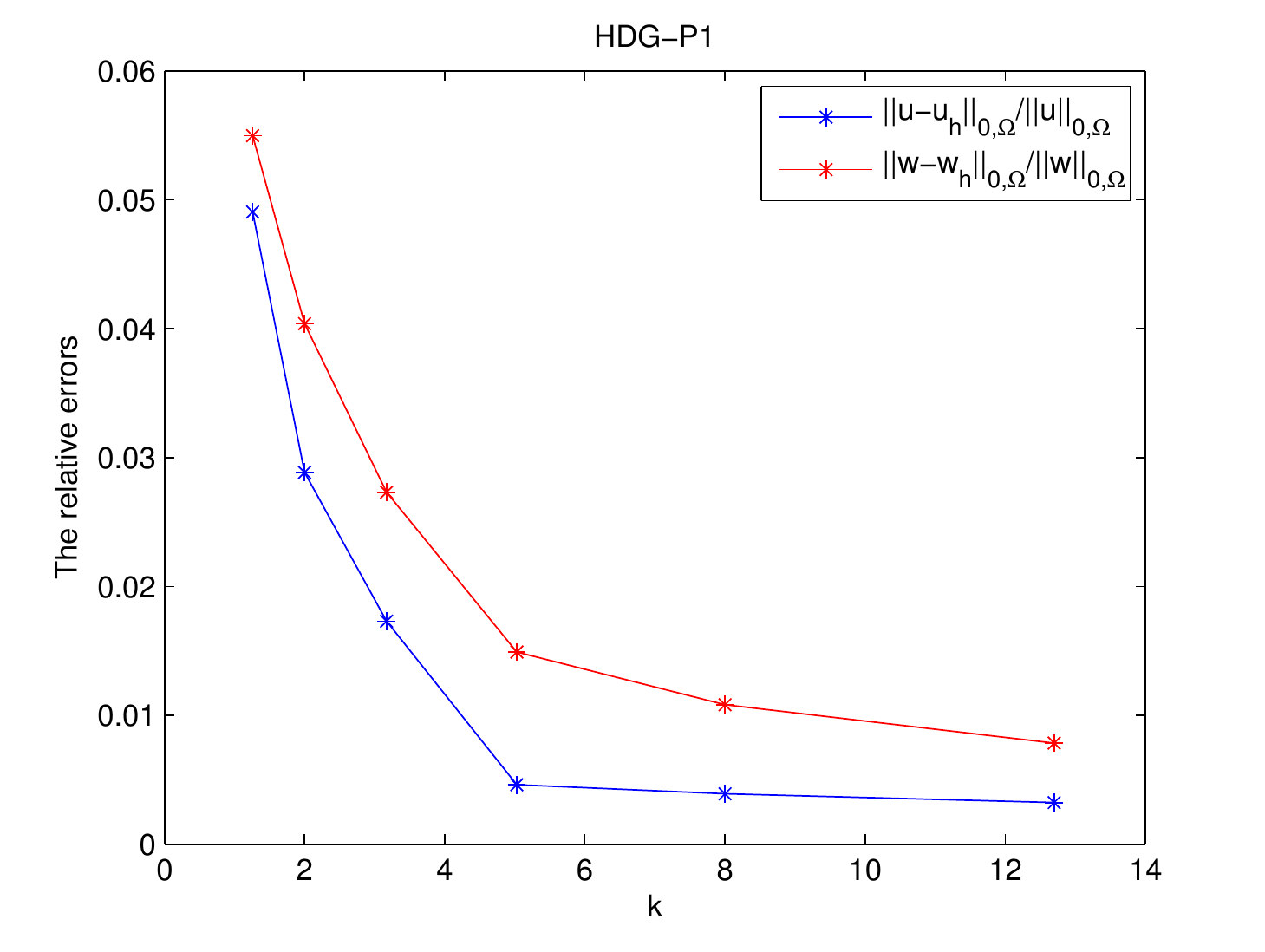}
   \caption{\footnotesize Left: The relative errors $\|\Bu-\Bu_h\|_{0,\Omega}/\|\Bu\|_{0,\Omega}$ and $\|\Bw-\Bw_h\|_{0,\Omega}/\|\Bw\|_{0,\Omega}$ for the HDG-P1 approximation under the mesh condition $\kappa h= 2$.
    Right: The relative errors $\|\Bu-\Bu_h\|_{0,\Omega}/\|\Bu\|_{0,\Omega}$ and $\|\Bw-\Bw_h\|_{0,\Omega}/\|\Bw\|_{0,\Omega}$ for the HDG-P1 approximation under the mesh condition  $\kappa^3 h^2= 2$.}\label{fig3}
\end{figure}

For fixed wave number $\kappa$, we show the relative error $\|\Bw-\Bw_h\|_{0,\Omega}/\|\Bw\|_{0,\Omega}$ for the HDG-P1 approximation with respect to the relative error $\|{\bf curl} (\Bu- \Bu_h)\|_{0,\Omega}/\|{\bf curl}\, \Bu\|_{0,\Omega}$ for the standard lowest-order edge element approximation of the second type. The left graph of Figure \ref{fig4}
displays the relative error of the HDG-P1 solution for $\kappa= 10, 20, 30$, while the right one shows the relative error for the same cases based on the standard lowest-order edge element
method. We find that the relative error for the HDG-P1 approximation stays around
100\% while the relative error for the standard edge element approximation oscillates around 100\% before they are less than 100\%, which confirms the stability property of our theoretical analysis for the HDG method and indicates that the HDG method is more stable
than the standard edge element method for the time-harmonic Maxwell problem  with high wave
number.

\begin{figure}[htbp]
\centering
    \includegraphics[width=2.8in,height=1.6in]{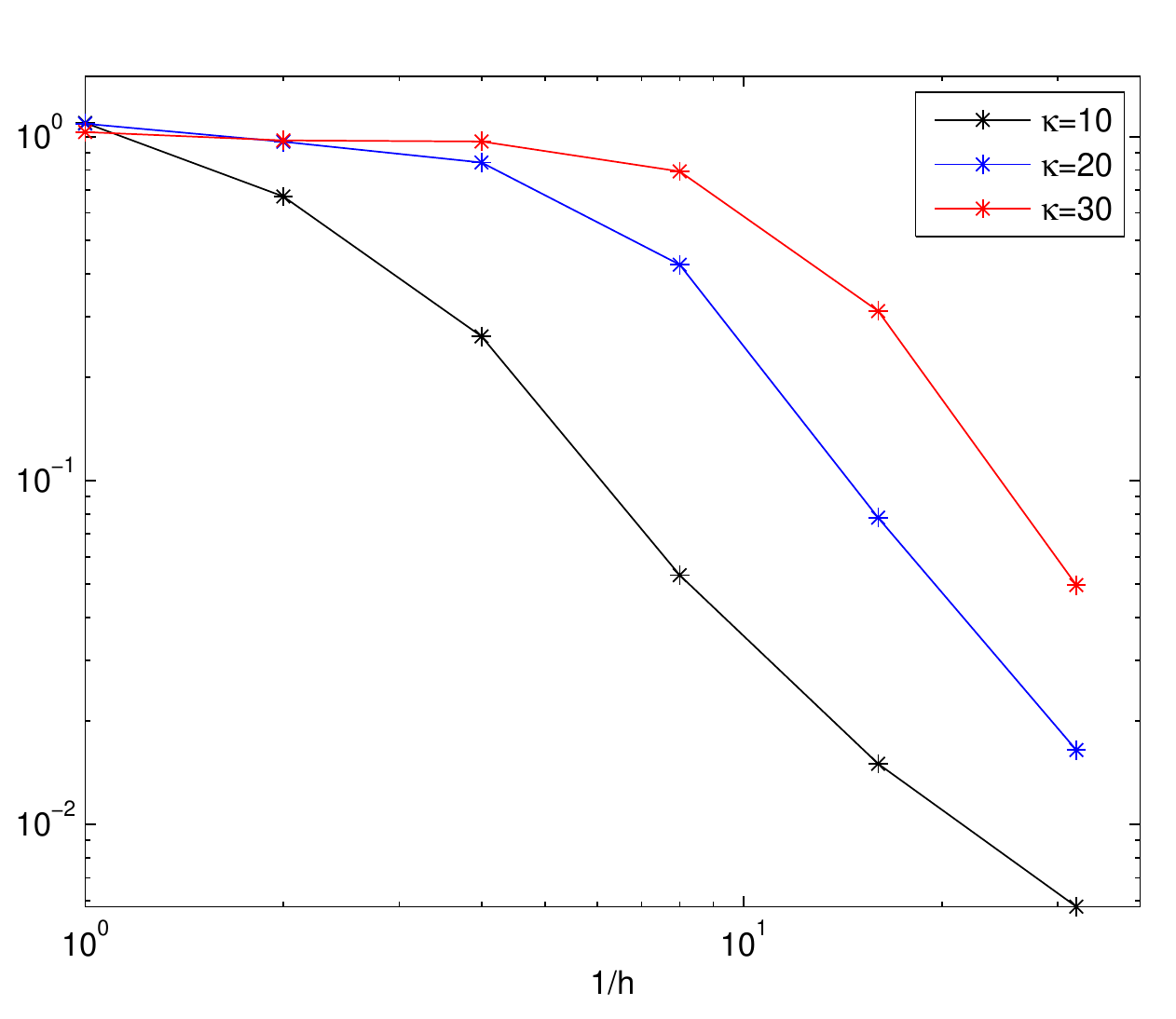}
    \includegraphics[width=2.8in,height=1.6in]{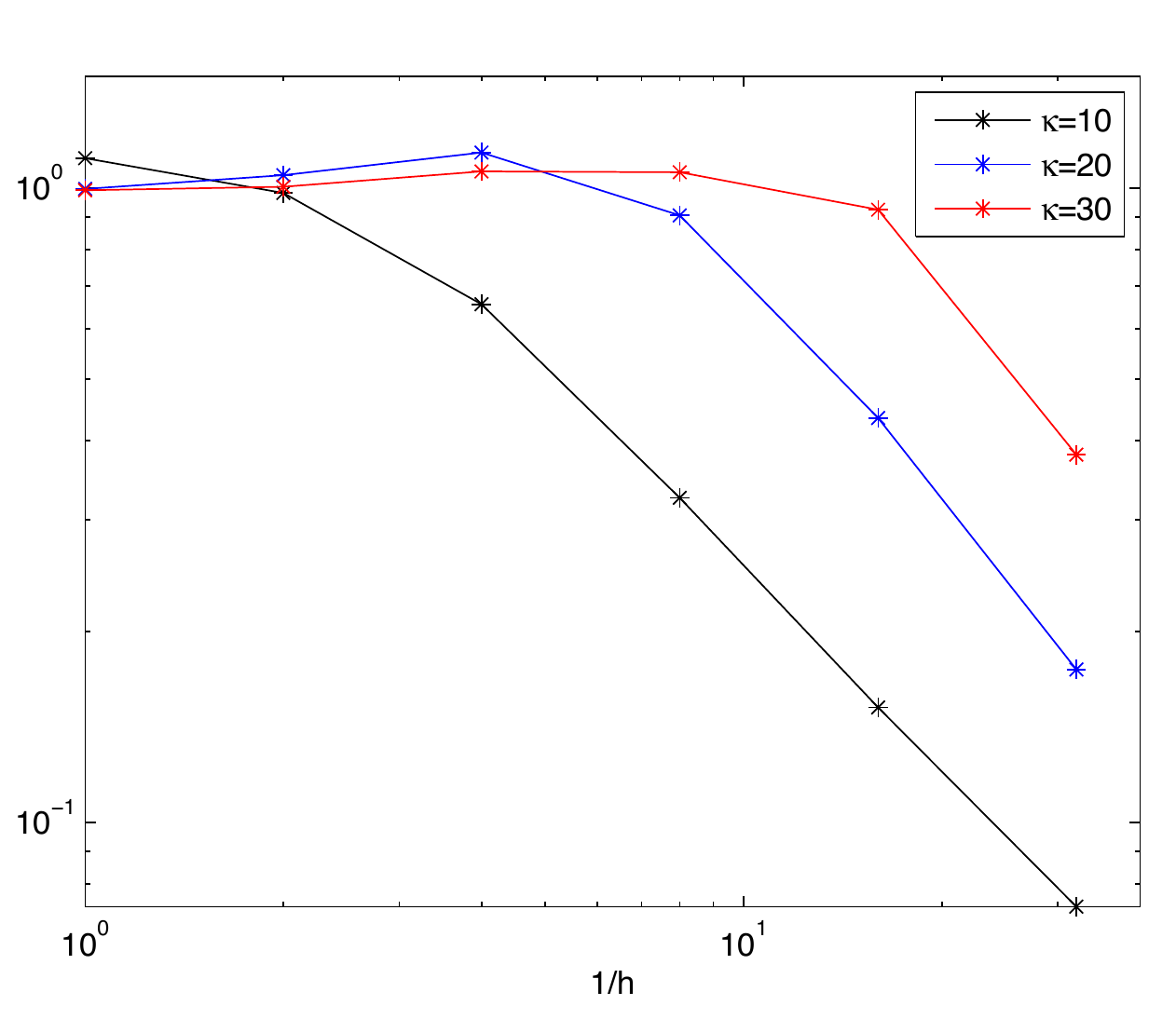}
    \caption{\footnotesize The relative error $\|\Bw-\Bw_h\|_{0,\Omega}/\|\Bw\|_{0,\Omega}$ (left) for the HDG-P1 approximation and the relative error $\|{\bf curl} (\Bu- \Bu_h)\|_{0,\Omega}/\|{\bf curl} \, \Bu\|_{0,\Omega}$ (right) for the lowest-order edge element (the second type) approximation for $\kappa=10,20,30$ respectively. }\label{fig4}
\end{figure}

\begin{table}[htbp]
\caption{\footnotesize The relative error $\|{\bf curl} (\Bu- \Bu_h)\|_{0,\Omega}/\|{\bf curl} \,\Bu\|_{0,\Omega}$ for the lowest-order edge element (the second type) approximation for the case $\kappa=50$ and the relative error $\|\Bw-\Bw_h\|_{0,\Omega}/\|\Bw\|_{0,\Omega}$ for the HDG-P1, HDG-P2 and HDG-P3 approximations with respect to different DOFs.}\label{tab1}
\begin{center}
\footnotesize
\begin{tabular}{|c|c|c|cccc|}
\hline
\multirow{3}{*}{Edge element}
 & \multicolumn{2}{|c|}{DOFs}                                  & 8368   & 62048 &  477376 &  3744128 \\
\cline{2-7}
 & \multicolumn{2}{|c|}{The relative error }    &  111.9\%     &   115.8\%    &   109.1\%   &  42.7\% \\
 \hline
\multirow{3}{*}{HDG-P1}
 & \multicolumn{2}{|c|}{DOFs}                                  & 9792   & 76032 &  599040 & 4755456   \\
\cline{2-7}
 & \multicolumn{2}{|c|}{The relative error }   &  96.8\%     &   96.7\%    &   82.3\%   & 30\%   \\
 \hline
\multirow{3}{*}{HDG-P2}
 & \multicolumn{2}{|c|}{DOFs}                                  & ---   & 19584 &  152064 & 1198080   \\
\cline{2-7}
 & \multicolumn{2}{|c|}{The relative error }   &  ---   &   100\%    &  89.4\%   & 21.6\%   \\
 \hline
\multirow{3}{*}{HDG-P3}
 & \multicolumn{2}{|c|}{DOFs}                                  & ---   & 32640 &  253440 & 1996800   \\
\cline{2-7}
 & \multicolumn{2}{|c|}{The relative error }   &  ---   &   100\%    &  54.3\%   & 2\%   \\
 \hline

\end{tabular}
\end{center}
\end{table}

Table \ref{tab1} shows the numbers of degrees of freedom (DOFs) and
the relative error $\|{\bf curl} (\Bu- \Bu_h)\|_{0,\Omega}/\|{\bf curl}\, \Bu\|_{0,\Omega}$ for
the standard edge element approximation with respect to the relative
error $\|\Bw-\Bw_h\|_{0,\Omega}/\|\Bw\|_{0,\Omega}$ for the HDG-P1, HDG-P2 and
HDG-P3 approximations. It can be observed that the HDG-P1 approximation performs
better than the standard edge element method when the numbers of DOFs are close. We
can also find that the HDG method with higher order polynomial approximation may
reach more accurate solutions with less DOFs, which indicates the
efficiency of the HDG method with high polynomial order for the
time-harmonic Maxwell problem with high wave number. We should
note that the numerical results in \cite{FW2014} show the stability
of the IPDG method based on the piecewise linear polynomial approximation for the time-harmonic Maxwell problem with high wave number. Here, our HDG method preserves the advantages of the IPDG method in \cite{FW2014}, and it results in a discrete system with significantly reduced DOFs when it is applied for the high order polynomial approximation.

\begin{figure}[htbp]
\includegraphics[width=3in,height=1.6in]{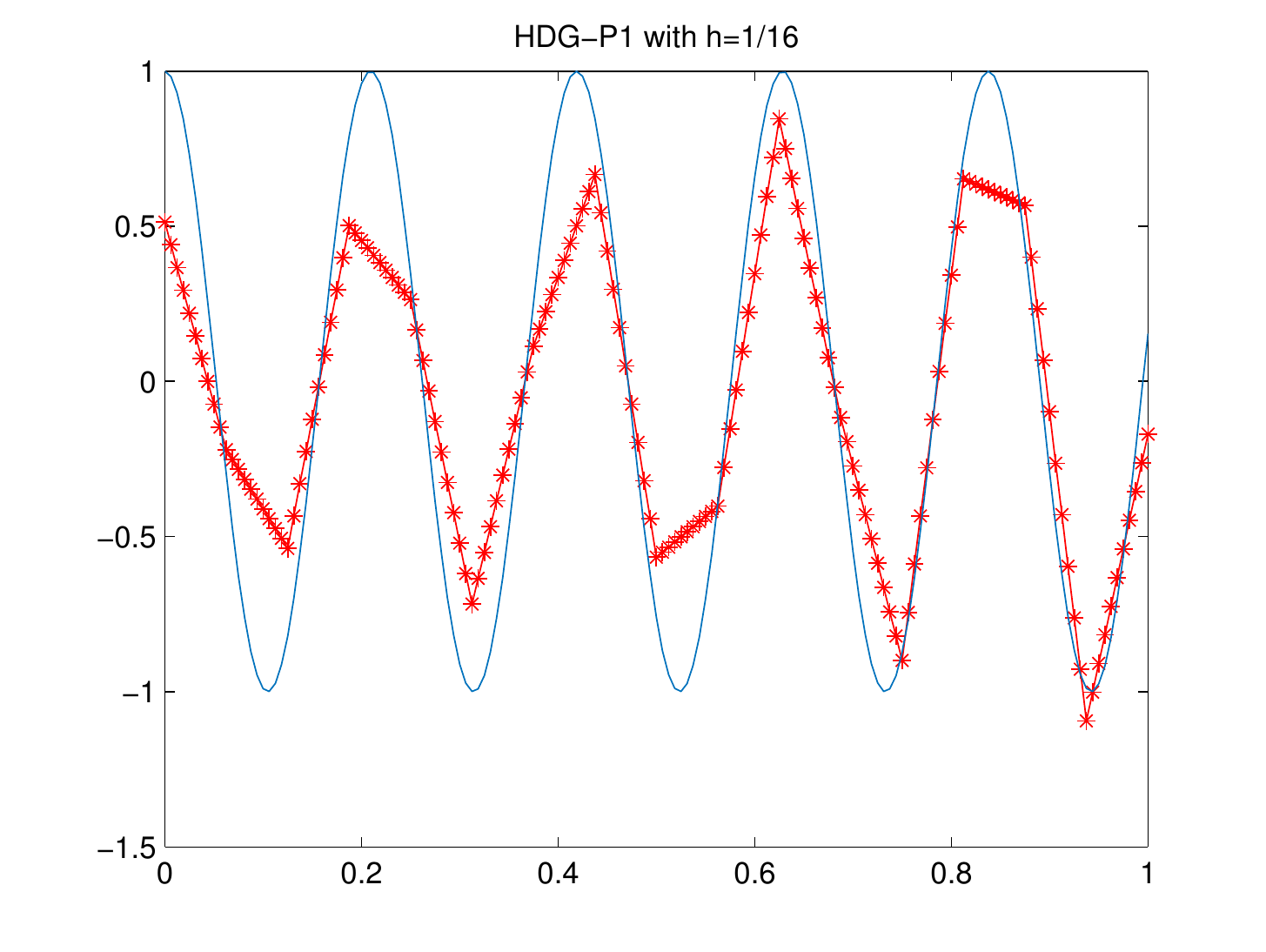}
\includegraphics[width=3in,height=1.6in]{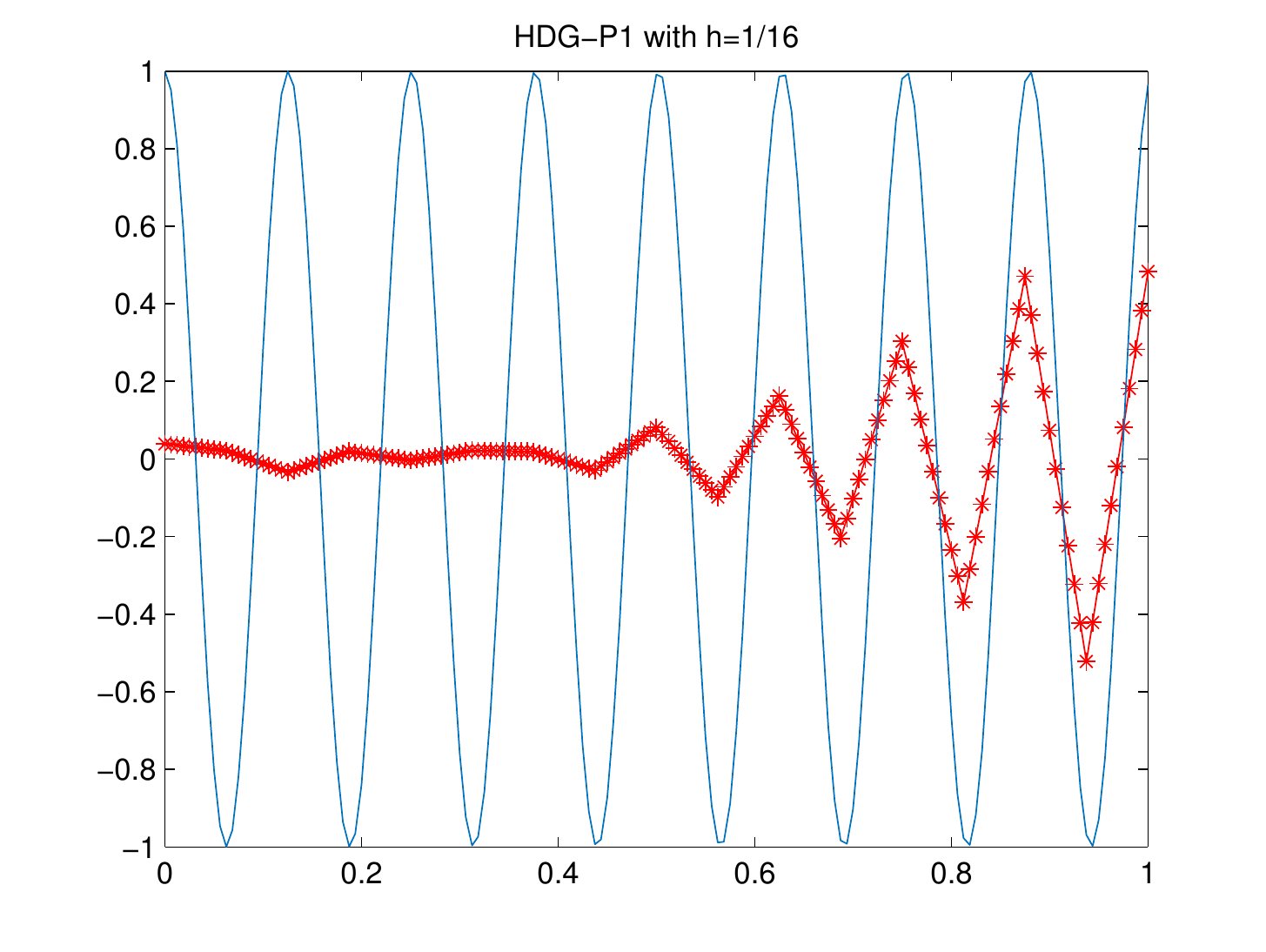}
\includegraphics[width=3in,height=1.6in]{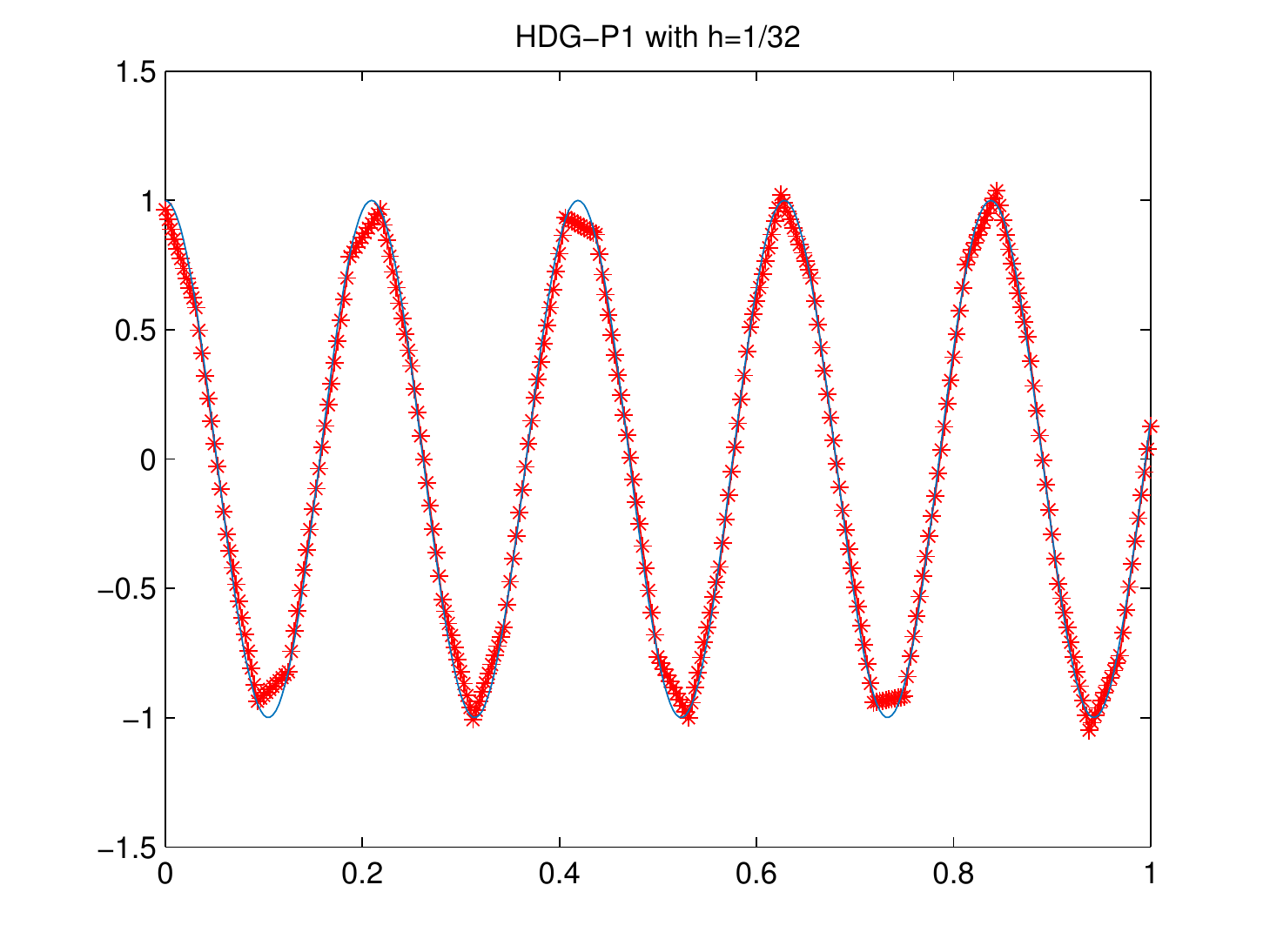}
\includegraphics[width=3in,height=1.6in]{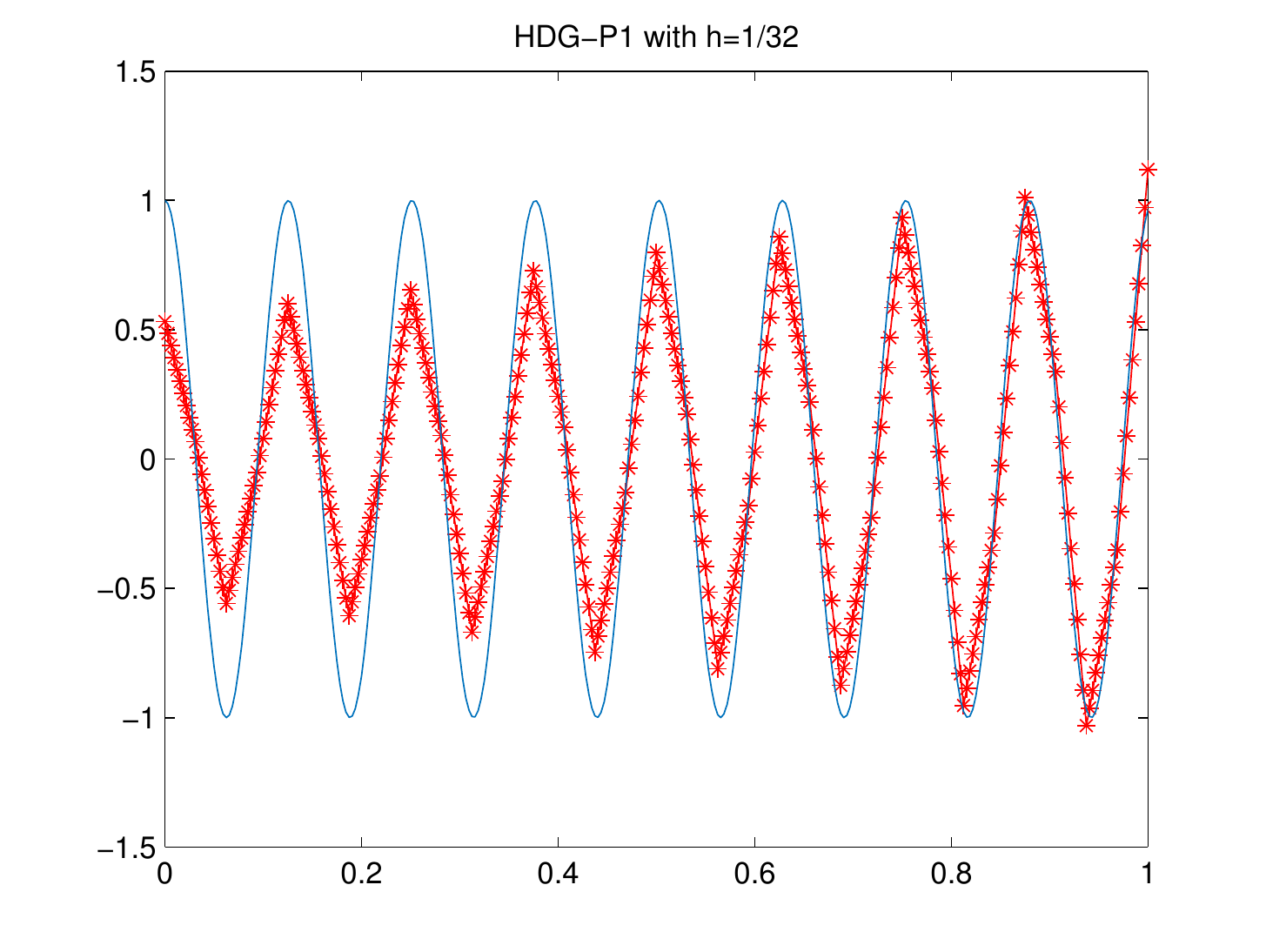}
\caption{\footnotesize The traces of the real part of the first
component of the HDG-P1 solutions for $\kappa=30$ and $\kappa=50$ (left and right) on the meshes with $h=1/16$ and $h=1/32$ (top and bottom). The traces of the real part of the first component of the exact solution are plotted in the blue lines.}\label{fig5}
\end{figure}

\begin{figure}[htbp]
\includegraphics[width=3in,height=1.6in]{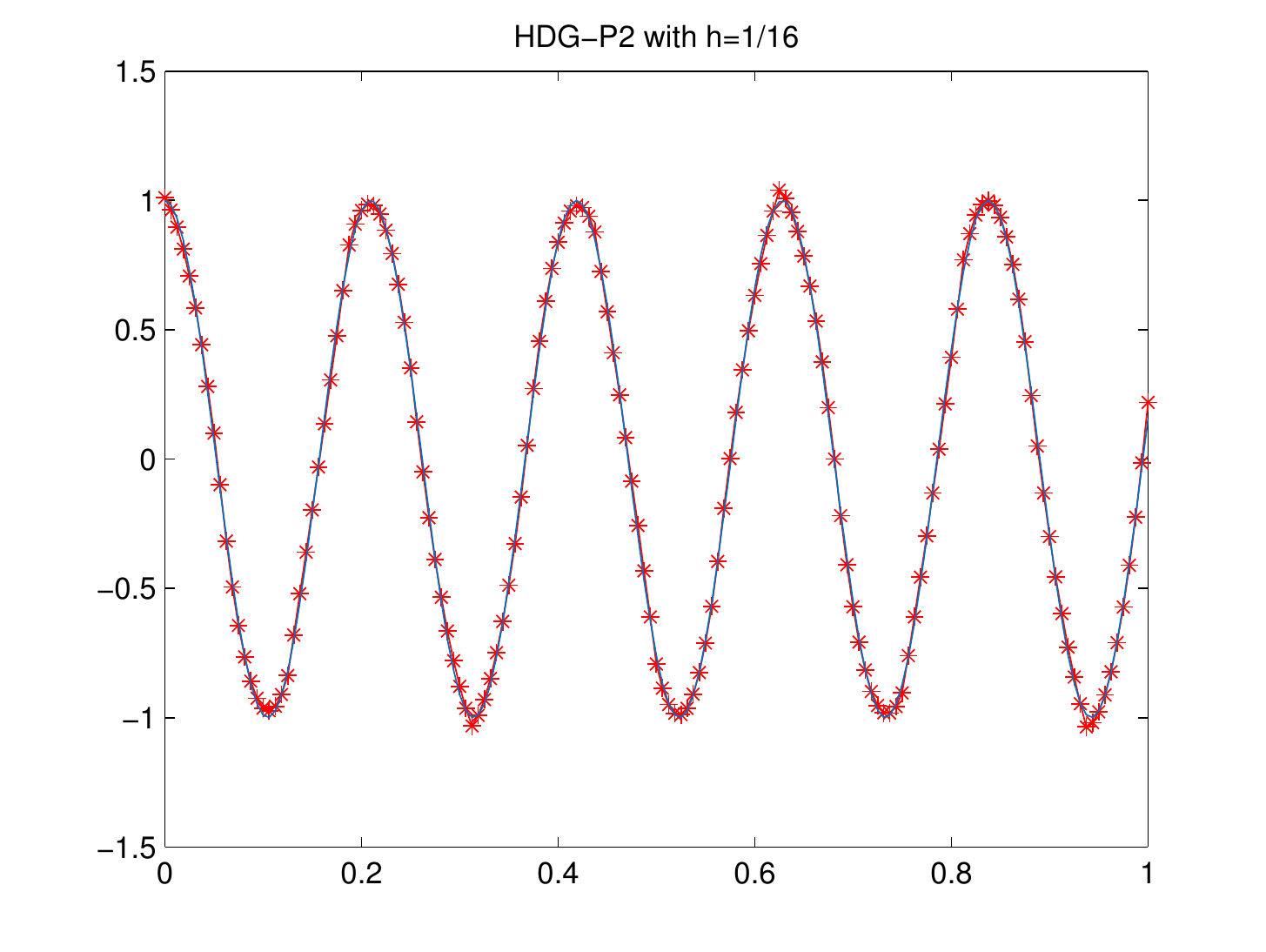}
\includegraphics[width=3in,height=1.6in]{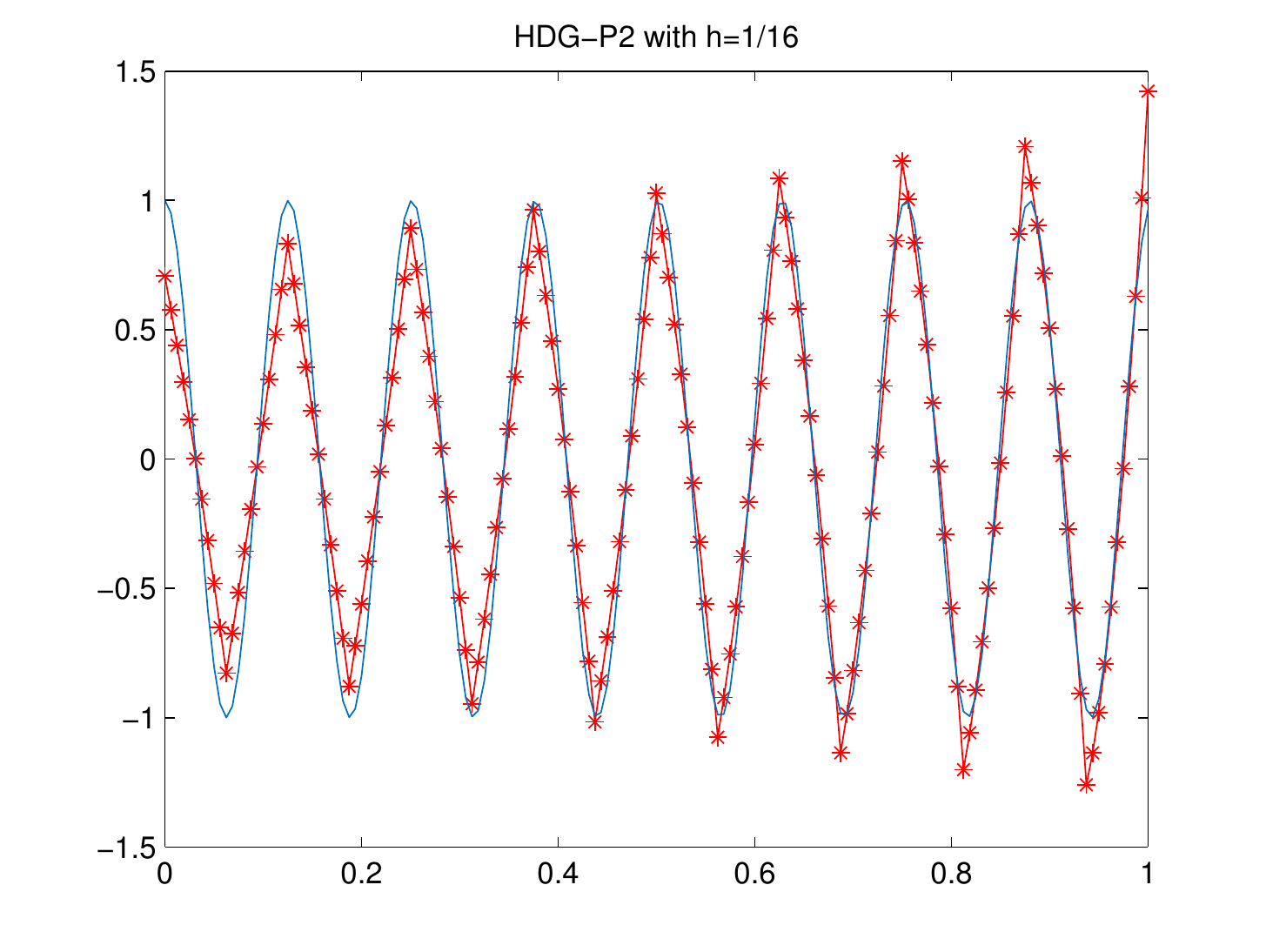}
\includegraphics[width=3in,height=1.6in]{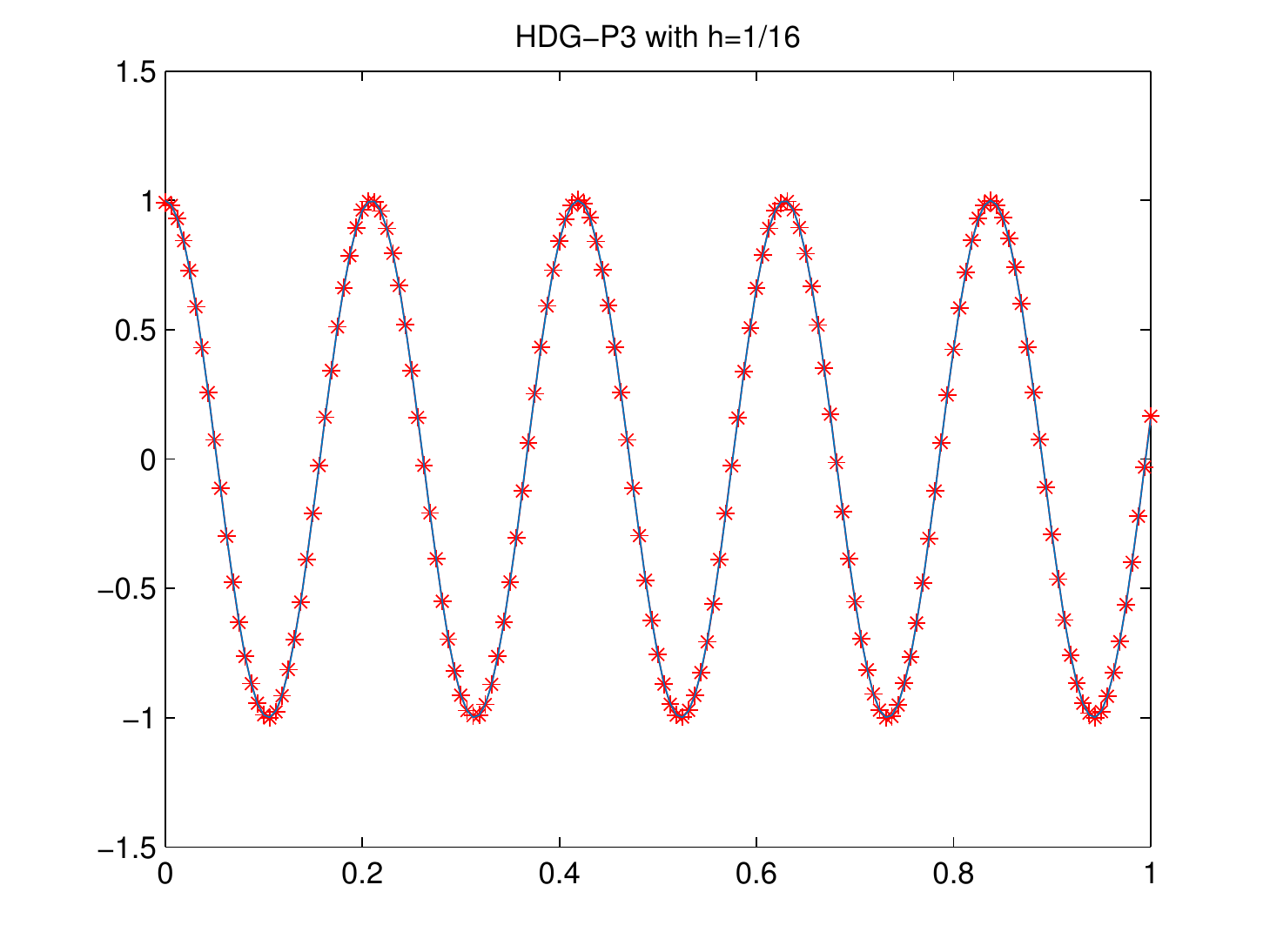}
\includegraphics[width=3in,height=1.6in]{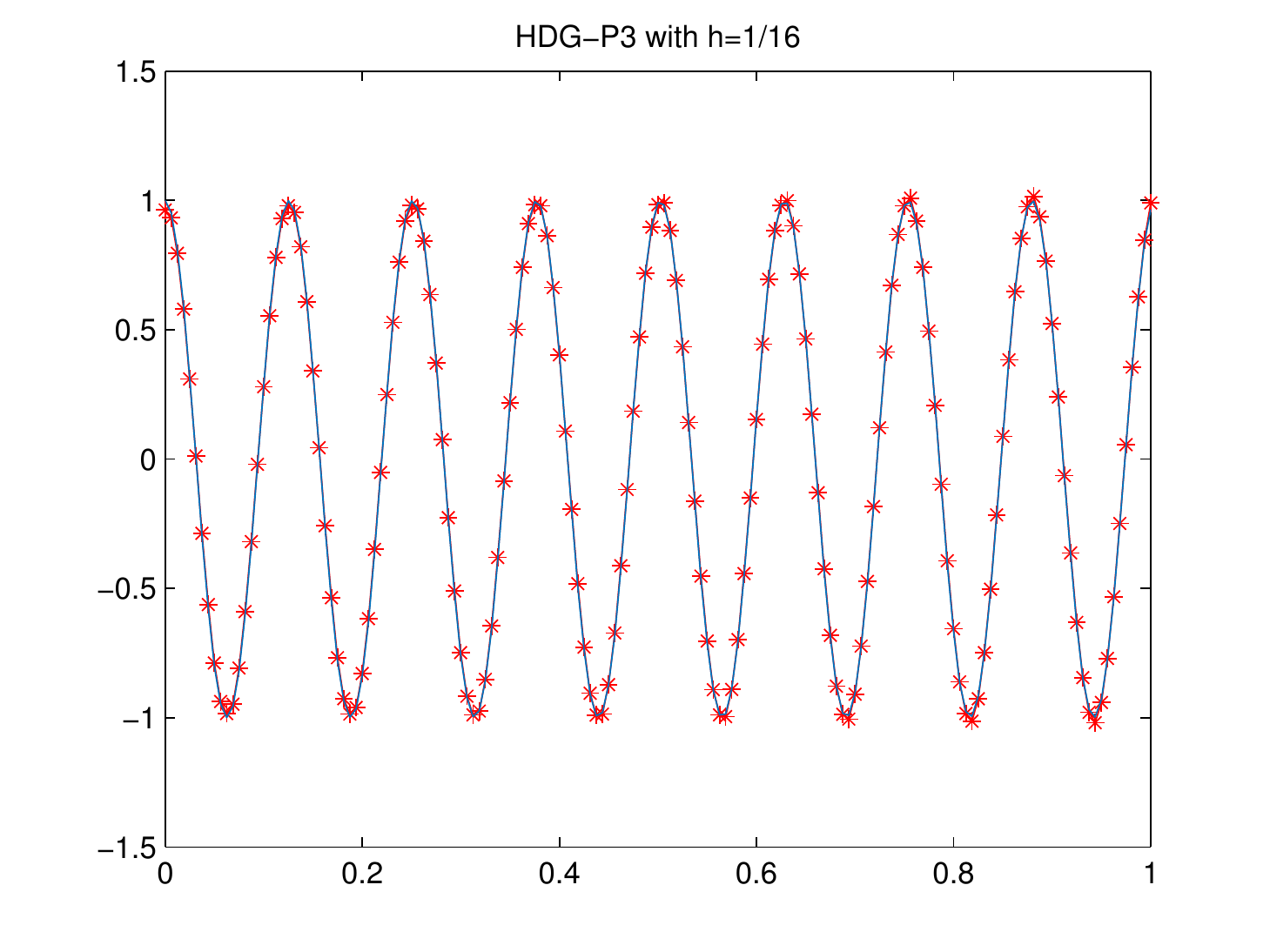}
\caption{\footnotesize The traces of the real part of the first
component of the HDG-P2 and HDG-P3 solutions (top and bottom) for $\kappa=30$ and $\kappa=50$ (left and right) on the mesh with $h=1/16$. The traces of the real part of the first component of the exact solution are plotted in the blue lines.}\label{fig6}
\end{figure}

For more detailed comparison between the HDG methods with different polynomial order approximations, we consider the problems with wave number $\kappa=30,50$.  We restrict the solution plot in the line segment $\{(x,y,z): x=0.5, y=0.5, 0\leq z
\leq 1\}$ and  observe the traces of the real part of the first
component of the HDG solutions. The traces of the real  part of the
first component of the exact solution are also  plotted in the blue
lines in Figure \ref{fig5} and Figure \ref{fig6}. The left graphs of Figure \ref{fig5}
display the traces of the real part of the first component of the
HDG-P1 solution on the meshes with $h=1/16$ and $h=1/32$ for $\kappa=30$, while the right graphs of Figure \ref{fig5} show the same traces for $\kappa=50$. Figure \ref{fig6} displays the traces of the real part of the first component of the HDG-P2 and HDG-P3 solutions on the mesh with $h=1/16$ for $\kappa=30,50$ (left, right). On the coarse mesh with $h=1/16$, the
shapes of the HDG-P2 and HDG-P3 solutions are roughly the same as the
exact solution while the shape of the HDG-P1 solution does not match
the exact solution well. We can also observe that the HDG solutions of
high order polynomial approximations on the mesh with $h=1/16$ perform even
better than the HDG-P1 solution on the mesh with $h=1/32$ especially for $\kappa=50$, which
shows the advantage of the HDG method with high order polynomial approximation for
the time-harmonic Maxwell problem with high wave number. Thus, although the phase error appears in the cases of coarse mesh and low order polynomial approximation, it can be reduced in the fine meshes or by high order polynomial approximations.

\section*{Acknowledgment} The authors are very grateful to the anonymous referees and the editor for their many valuable comments and suggestions that led to an improved presentation of this paper.

\end{document}